\documentclass{amsart}
\usepackage[paper=a4paper, margin=2.8cm]{geometry}

\usepackage[english]{babel}

\usepackage{mathrsfs, mathtools, amssymb, mathabx, stmaryrd, enumitem}
\usepackage{caption}
\usepackage{subcaption}
\usepackage{dsfont}
\usepackage{comment}

\usepackage[all,cmtip]{xy}
\usepackage{tikz}
\usepackage{tikz-3dplot}
\usetikzlibrary{calc}

\tdplotsetmaincoords{70}{120}

\usepackage{hyperref}
\hypersetup{
    unicode=false,          
    pdftoolbar=true,        
    pdfmenubar=true,        
    pdffitwindow=false,     
    pdfstartview={FitH},    
    pdftitle={},    
    pdfauthor={},     
    pdfkeywords={}, 
    pdfnewwindow=true,      
    colorlinks=true,       
    linkcolor=blue,          
    citecolor=red,        
    filecolor=magenta,      
    urlcolor=cyan           
}

\theoremstyle{plain}
\newtheorem{theorem}{Theorem}[section]
\newtheorem*{theoremM}{Theorem}
\newtheorem{lemma}[theorem]{Lemma}
\newtheorem{proposition}[theorem]{Proposition}
\newtheorem{corollary}[theorem]{Corollary}
\newtheorem{assumption}{Assumption}

\newtheorem*{assumptionI}{Irreducibility Assumption}
\newtheorem*{assumptionS}{Stratification Assumption}

\theoremstyle{definition}
\newtheorem{definition}[theorem]{Definition}
\newtheorem{remark}[theorem]{Remark}
\newtheorem{example}[theorem]{Example}

\newcommand{\rleft}{\mathopen{}\mathclose\bgroup\left}
\newcommand{\rright}{\aftergroup\egroup\right}

\newcommand{\C}{{\mathbb{C}}}

\newcommand{\R}{{\mathbb{R}}}
\newcommand{\Z}{{\mathbb{Z}}}

\newcommand{\HH}{\mathrm{H}}

\newcommand{\cI}{{\mathcal{I}}}

\newcommand{\rk}{{\mathrm{rk}}}

\newcommand{\cP}{{\mathcal{P}}}

\newcommand{\rev}{{\mathrm{rev}}}

\DeclareMathOperator{\verti}{vert}
\def\OO{\mathcal{O}}
\def\cO{\mathcal{O}}
\def\BB{\mathcal{B}}
\def\RR{\mathbb{R}}
\def\CC{\mathbb{C}}
\DeclareMathOperator{\CH}{CH}
\DeclareMathOperator{\Con}{Con}

\def\CCC{\mathcal{C}}

\newcommand\f{{\mathsf{f}}}

\begin{document}
\selectlanguage{english}



\title[Vertex Posets, Monotone Path Polytopes, and Chow Polynomials]{Vertex Posets, Monotone Path Polytopes, and Chow Polynomials}

\author{Mateusz Micha\l ek}
\address[M.~Micha\l ek]{University of Konstanz, Department of Mathematics and Statistics,
Post Box 192,
78457 Konstanz, Germany}
\email{mateusz.michalek@uni-konstanz.de}

\author{Leonid Monin}
\address[L.\,Monin]{Institut de Math\'ematiques, EPFL, B\^atiment MA, Station 8, 1015 Lausanne, Switzerland}
\email{leonid.monin@epfl.ch}

\author{Botong Wang}
 \address[B.~Wang]{Department of Mathematics,
         University of Wisconsin--Madison,
         Van Vleck Hall,
         480 Lincoln Drive, Madison, WI
       }
\email{wang@math.wisc.edu}


\begin{abstract}
Let $P\subset\mathbb R^n$ be a convex polytope and let $\ell$ be a linear functional which is nonconstant on every edge of $P$. The induced acyclic orientation determines positive and negative Bia{\l}ynicki-Birula type partitions of $P$ into unions of relative interiors of faces. 
Our first result establishes a duality: the positive partition is a stratification if and only if the negative one is a stratification.


Our second result connects poset invariants with monotone path polytopes. Assuming the induced vertex relation admits the structure of a graded poset, we prove that the Chow polynomial of the resulting vertex poset agrees with the $h$-polynomial of a (dual) monotone path polytope.

\end{abstract}

\maketitle


\section{Introduction}
\label{sec:intro}

Let $P\subset \mathbb R^n$ be a convex polytope and let $\ell:\mathbb R^n\to \mathbb R$ be a linear functional which is nonconstant on every edge of $P$. Then $\ell$ induces an acyclic orientation of the $1$-skeleton of $P$ by directing each edge from the endpoint of smaller $\ell$-value to the endpoint of larger $\ell$-value. We denote by $\hat 0$ (resp.~$\hat 1$) the unique vertex minimizing (resp.~maximizing) $\ell$ on $P$.

A basic object attached to $(P,\ell)$ is a Bia{\l}ynicki-Birula type partition of $P$ into subsets indexed by the vertices. For a vertex $v$, let $O^-(v)$ (resp.~$O^+(v)$) be the union of relative interiors of faces $G$ of $P$ on which $\ell$ attains its maximum (resp.~minimum) at $v$ (see Definition~\ref{def:cell}). Intuitively, $O^-(v)$ collects points whose ``sink'' (within the minimal face containing the point) is $v$, while $O^+(v)$ collects points whose ``source'' is $v$.

Although $\{O^-(v)\}_v$ and $\{O^+(v)\}_v$ are always partitions of $P$, they need not be \emph{stratifications} in the topological sense: it may happen that the closure of two strata intersect without one being contained in the other. Our first goal is to understand the combinatorial structure on the vertex set which determines whether the BB-type partitions are stratifications. 

There are several natural relations on the vertices induced by $\ell$. First, the directed $1$-skeleton defines the \emph{chain order} $\mathcal C$ by declaring $\mathcal C(v,w)$ if there exists a directed chain $v\to\cdots\to w$. Second, we consider the \emph{witness relation} $\mathcal O$, where $\mathcal O(v,w)$ holds if there exists a face $G$ of $P$ such that $\ell$ attains its minimum on $G$ at $v$ and its maximum on $G$ at $w$. Third, we introduce positive and negative Bruhat-type relations $\mathcal B^\pm$ (see Section~\ref{sec:polytope}). In general, $\mathcal O$ need not be transitively closed, so it may fail to be a partial order, even though its transitive closure is always $\mathcal C$.

Our first main result is a symmetry statement that is not apparent from the definitions. In particular, the statement is no longer true without the genericity assumption for $\ell$.

\begin{theoremM}[Theorem~\ref{thm: O- strat iff O+}]
The partition $O^-$ is a stratification if and only if the partition $O^+$ is a stratification.
\end{theoremM}

The structure becomes significantly simpler after imposing an irreducibility condition motivated by geometry. Let $F^+(v)$ and $F^-(v)$ be the closure of $O^+(v)$ and $O^-(v)$, respectively. 
In general $F^\pm(v)$ can be a union of faces; we single out the case when this does not happen.

\begin{assumptionI}
For every vertex $v$, both $F^-(v)$ and $F^+(v)$ are faces of $P$.
\end{assumptionI}

Under Irreducibility Assumption we obtain a collection of equivalent characterizations of the stratification property (Theorem~\ref{thm:equivalences}). In particular, $O^-$ and $O^+$ are stratifications  if and only if the witness relation $\mathcal O$ is an order relation (equivalently, $\mathcal O=\mathcal C$). In this case, the resulting vertex poset is graded with rank function
\[
\rho(v)=\dim F^-(v).
\]

A second focus of the paper is the \emph{monotone path polytope} $\CH(P)=\CH_\ell(P)$, which is a special case of a fiber polytope introduced in \cite{billera1992fiber}. Monotone path polytopes were extensively studied, see for instance \cite{athanasiadis1999piles,athanasiadis2000monotone,black2023polyhedral,poullot2024vertices}. In the projective toric setting (when $P$ is a very ample lattice polytope and $\ell$ has rational slope), $\CH(P)$ describes the normalization of the Chow quotient by the $\CC^*$-action determined by $\ell$, as observed in \cite{billera1992fiber, kapranov1991quotients}. In general, faces of $\CH(P)$ are described by sequences of faces of $P$ satisfying additional compatibility constraints (see Theorem~\ref{thm:facesCH}).

We then restrict attention to the regime in which the induced vertex relations behave well throughout.

\begin{assumptionS}
In addition to the Irreducibility Assumption, we assume that $O^-$ is a stratification, or equivalently, any of the equivalent conditions of Theorem~\ref{thm:equivalences} is satisfied.
\end{assumptionS}

Under the Stratification Assumption, the faces of $\CH(P)$ admit a dramatically simpler description: the additional compatibility condition in the general Billera--Sturmfels framework becomes redundant (Theorem~\ref{thm:facesofCH}).
We further characterize when $\CH(P)$ is simple (Theorem~\ref{thm:simpleCH}). 
As a consequence, if $P$ is simple and $O^-$ is a stratification, then $\CH(P)$ is simple (Corollary~\ref{cor:Psim then CH sim}). 

The final section connects these polyhedral structures back to poset invariants. Motivated by poset theory—specifically Stanley’s theory of kernels and Kazhdan--Lusztig--Stanley polynomials \cite{stanley1992subdivisions} and the recent notion of Chow functions/polynomials of posets \cite{ferroni2024chow,Proudfoot}—we obtain a concrete bridge between combinatorics and geometry in the simple $\CH(P)$ setting: we exhibit a natural kernel on the vertex poset and show that the resulting Chow polynomial computes the virtual Poincar\'e polynomial of the Chow quotient, identified as the $h$-polynomial of a dual monotone path polytope (Theorem~\ref{thm:Chow poly is h of Chow}).

\subsection*{Geometric motivations}
Bia{\l}ynicki-Birula in his fundamental work \cite{bialynicki1973some} introduced canonical decompositions of algebraic varieties endowed with a $\CC^*$--action. When the fixed point set is finite, each fixed point $v$ has an attracting cell consisting of those $x$ for which $\lim_{t\to\infty} t\cdot x=v$, and a repelling cell defined using $t\to 0$. For a projective toric variety $X_P$ associated to a polytope $P$, choosing such a one--parameter subgroup is equivalent to choosing a linear functional $\ell$ that is non--constant on edges. In this case the torus fixed points are the vertices of $P$, and the Bia{\l}ynicki-Birula cells coincide (via the moment map) with our polyhedral cells $O^-(v)$ and $O^+(v)$. Thus the stratification problem for $O^\pm$ is the toric avatar of the question when BB decompositions behave well.

The witness relation $\mathcal O$ and its transitive closure $\mathcal C$ have a direct geometric meaning: $\mathcal O(v,w)$ holds precisely when there exists a $\CC^*$--orbit whose closure connects the fixed points $v$ and $w$, and $\mathcal C$ records chains of such orbits. In this sense our main symmetry theorem $O^-$ stratifies $\Leftrightarrow O^+$ stratifies reflects a duality between attracting and repelling BB cells.

Beyond toric varieties, $\CC^*$--actions on flag varieties and Grassmannians lead to Bruhat and Richardson stratifications. For suitable actions the BB cells coincide with Schubert cells, motivating our introduction of Bruhat--type relations on vertices. Classically, Chow quotients of Grassmannians by $\CC^*$--actions produce important spaces such as the variety of complete collineations \cite{Thaddeus2}. Its toric counterpart is the fact that the Chow quotient of $(\mathbb P^1)^n$ by the diagonal $\CC^*$--action is the permutohedral variety \cite{billera1992fiber}, i.e. the toric variety of the monotone path polytope of the cube. Via the natural embedding $(\mathbb P^1)^n\hookrightarrow\mathrm{Gr}(n,2n)$ this identifies the permutohedral variety as a subvariety of the variety of complete collineations \cite{michalek2023enumerative, dinu2025applications}.

These ideas play a central role in recent work on compactifications of linear spaces. In particular, closures of linear subspaces in $(\mathbb P^1)^n$---arrangement Schubert varieties \cite{ardila2016closure}---may be singular, but their Chow quotients by suitable $\CC^*$--actions are the smooth wonderful models of matroids, a phenomenon highlighted in the groundbreaking work of Huh \cite{huh2014rota}. This striking improvement of singularities motivated our study of when the monotone path polytope $\CH(P)$ is simple. More generally, $\CC^*$--actions and their Chow quotients have been intensively studied in recent years, see for instance \cite{Kapranov, keel2004chow, gibney2011equations,  baker2015chow,OSCRW,michalek2021maximum,manivel2023complete,GmChowquotient,knutson2025stable}.


\subsection*{Organization of the paper}
Section~\ref{sec:polytope} introduces the relations on vertices and proves various equivalence conditions of the stratification assumption (Theorem~\ref{thm: O- strat iff O+} and Theorem~\ref{thm:equivalences}). Section~\ref{sec:Chowpolytope} studies monotone path polytopes under the Stratifcation Assumption and proves the simplicity criterion (Theorem~\ref{thm:simpleCH}). 
Section~\ref{sec:ker and KLS poly} develops the kernel/Kazhdan--Lusztig--Stanley and Chow polynomial consequences culminating in Theorem~\ref{thm:Chow poly is h of Chow}.

\vskip7pt
Independently and developed in parallel with our work, stratifications arising from Bia{\l}ynicki-Birula decompositions were studied in \cite{TeddyChayim}. That paper focuses on smooth varieties with $\CC^*$-action. While the two works overlap in certain cases---typically under additional assumptions---their results are complementary in general. 

\subsection*{Acknowledgements}
We thank the Institute for Advanced Study, where this project was initiated, and June Huh for their hospitality. We are also grateful to Lorenzo Vecchi and Tao Gui for helpful discussions. We thank Teddy Gonzales and Chayim Lowen for sharing their preprint and insightful ideas. In particular, following their observation, we replaced the section proving Remark \ref{rem:product of chains} with the closely related result Proposition \ref{prop:2faces=>simple}, which follows from \cite{Wiemeler, YuMasuda}.

Mateusz Micha\l ek  was partially supported by the Charles Simonyi Endowment at the Institute for Advanced Study and by the DFG grant 580118961. Botong Wang was partially supported by the NSF grant DMS-1926686.
Leonid Monin was partially supported by the DFG project number 539974215 and by SNSF grant 224099.

\section{Polytopes, stratifications, orders}\label{sec:polytope}

\subsection{Bruhat relations and stratifications}
Let $P\subset \R^n$ be a convex polytope. Let $\ell$ be a linear function on $\R^n$ that is nonconstant on the edges of $P$. We direct each edge of $P$ from the vertex on which $\ell$ attains smaller value to the vertex with larger value. We call the vertex at which $\ell$ attains minimal value over $P$ the \emph{source} and maximal value the \emph{sink}. For an edge $(v, w)$ of $P$, if $\ell(w)>\ell(v)$, we call $(v, w)$ a directed edge and denote by $v\rightarrow w$. 

\begin{definition}\label{def:cell}
\leavevmode
\begin{enumerate}[leftmargin=*]
    \item For each vertex $v$ of $P$ we define $O^-_P(v)$ (resp.~$O^+_P(v)$) as the union of relative interiors of all faces $G$ of $P$ such that $\ell$ attains its maximum (resp.~minimum) over $G$ at $v$.
    \item We define $F^-_P(v):=\overline{O^-_P(v)}$ and $F^+_P(v):=\overline{O^+_P(v)}$. Here, and throughout the paper, we use an overline to denote the closure. As the closure of a finite union of sets is the union of their closures, we see that $F^-_P(v)$ is the union of all faces $G$ such that $\ell$ attains its maximum on $G$ at $v$. 
    \item We define a binary relation $\OO_P$ on vertices, where $\OO_P(v,w)$ if and only if there exists a face $G$ of $P$ such that $\ell$ attains its maximum at $w$ and its minimum at $v$. We call such $G$ a \emph{witness} for $\OO_P(v,w)$. A witness does not have to be unique.
    \item We also define a \emph{negative} (resp. \emph{positive}) \emph{Bruhat relation} $\BB^-_P(v,w)$ (resp.~$\BB^+_P(v,w)$) by $v\in F^-_P(w)$ (resp. $w\in F^+_P(v)$). 
    \item Finally, we define the \emph{chain order} $\CCC_P$ to be the reflexive transitive closure of the relation given by having a directed edge. In other words, $\CCC_P(v, w)$ if there exists a chain of directed edges $v=v_0\rightarrow \cdots \rightarrow v_k=w$.
\end{enumerate}
 For $\OO_P$, $\BB_P^{\pm}$, $\CCC_P$, $F^{\pm}_P(v)$  and $O^{\pm}_P(v)$, we will often omit $P$ from the notation. 
\end{definition}
We start with a few basic observations about the relations introduced above. 
\begin{lemma}\label{lem:path}
Let $v$ be a vertex of a face $G$ of $P$. If $\ell$ attains its maximum on $G$ at $w$, then $\CCC(v,w)$.
\end{lemma}
\begin{proof}
    As long as a vertex of $G$ is not $w$, it must have an outgoing edge in $G$. Thus, we can consecutively build a chain of edges  from $v$ and such a chain must finish at $w$.
\end{proof}
\begin{lemma}\label{lem:C=closO}
    The order $\CCC$ is the transitive closure of $\OO$.
\end{lemma}
\begin{proof}
   Clearly, $\CCC$ is an order relation. By Lemma \ref{lem:path}, we have $\OO\subset \CCC$. Thus, the transitive closure of $\OO$ is also contained in $\CCC$. Conversely, notice that for any directed edge $v\to w$, we have $\OO(v,w)$. Thus, $\CCC$ is contained in the transitive closure of $\OO$. 
\end{proof}

\begin{lemma}\label{lem:easyorderinclusions}
    We have $\OO\subseteq \BB^-\cap \BB^+$ and $\BB^-,\BB^+\subseteq \CCC$. 
\end{lemma}
\begin{proof}
    Let $G$ be a witness for $\OO(v,w)$. Then $v,w\in G\subset F^-(w)\cap F^+(v)$, which proves the first inclusion. 
If $\BB^-(v,w)$, then $v\in F^-(w)$, i.e.,~there exists a face $G$ in $F^-(w)$ containing $v$ and the function $\ell$ attains its maximum on $G$ at $w$. Applying Lemma \ref{lem:path} to $G$, we obtain $\BB^-\subset\CCC$. The proof of  $\BB^+\subset\CCC$ is analogous.
\end{proof}

Note that $O^-(v)$, as $v$ varies over the vertices of $P$, partitions $P$ into disjoint subsets, each containing precisely one vertex. We will refer to the partition $\{O^-(v)|v\in \verti P\}$ simply by $O^-$.
Recall that a partition of a topological space $X=\bigcup S_i$ is a \emph{stratification} if we have an implication $S_i\cap \overline{S_j}\neq\emptyset\Rightarrow {S_i}\subset \overline{S_j}$. In general, the restriction of a stratification to a closed subset may fail to be a stratification.
\begin{example}
    Let $X=[0,1]$ with stratification $S_1=(0,1)$ and $S_2=\{0,1\}$. Let $Y=[0,\frac{1}{2}]\cup \{1\}$. We see that $S_1\cap Y, S_2\cap Y$ do not form a stratification of $Y$.
\end{example}
\begin{lemma}\label{lem:induced stratification}
    Let $\{S_i\}$ be a stratification of a topological space $X$ and let $Y$ a be closed subset of $X$. Suppose that if  $S_i\cap Y\neq \emptyset$, then $\overline{S_i\cap Y}=\overline{S_i}\cap Y$. Then $\{S_i\cap Y\}$ is a stratification of $Y$.
\end{lemma}
\begin{proof}
    Suppose $(S_i\cap Y)\cap \overline{S_j\cap Y}\neq \emptyset$. In particular, $S_i$ must intersect $\overline{S_j}$ and thus $S_i\subset \overline{S_j}$. It follows that 
    \[
    S_i\cap Y\subset \overline{S_j}\cap Y.
    \]
    Since $S_j\cap Y\neq \emptyset$, we further have $\overline{S_j}\cap Y=\overline{S_j\cap Y}$. Therefore, 
    \[
    S_i\cap Y\subset \overline{S_j\cap Y},
    \]
    which completes the proof. 
\end{proof}
Under the assumptions of the lemma, we call $\{S_i\cap Y\}$ the \emph{induced stratification} on $Y$.
\begin{lemma}\label{lem:F-intersect gives F-}
    For any face $G$ of $P$ and any vertex $w$ of $G$ we have: 
    \begin{enumerate}
    \item $O^-_G(w)=O^-_P(w)\cap G$,
    \item $F^-_G(w)=F^-_P(w)\cap G$.
    \end{enumerate}
In particular, the assumption of Lemma \ref{lem:induced stratification} holds for the stratification $O^-$ of $P$ and the closed subset $G\subset P$. 
\end{lemma}
\begin{proof}
For (1), both sides are equal to the union of relative interiors of faces $F\subset G$ such that $\ell$ attains its maximum on $F$ at $w$. 

For (2), by the first equality, we have
\[
F^-_G(w)=\overline{O^-_G(w)}=\overline{O^-_P(w)\cap G}\subset \overline{O^-_P(w)}\cap G=F^-_P(w)\cap G.
\]
To show $F^-_P(w)\cap G\subset F^-_G(w)$, choose a face $Q$ of $P$ such that $\ell$ attains maximum on $Q$ at $w$. We need to prove that $Q\cap G\subset F^-_G(w)$. Since $w\in G$, $\ell$ attains maximum on $Q\cap G$ at $w$. Since $Q\cap G$ is a face of $G$, it follows that $Q\cap G\subset F^-_G(w)$. 
%
%
\end{proof}
As a corollary we obtain the following. 
\begin{corollary}\label{cor:stratonface}
    If $O^-$ is a stratification on $P$, then it also induces a stratification on every face of $P$.
\end{corollary}

\begin{lemma}
    If $O^-$ is a stratification, then $\BB^-$ coincides with $\CCC$. In particular, $\BB^-$ is an order relation. 
\end{lemma}
\begin{proof}
  By Lemma \ref{lem:easyorderinclusions}, we only have to prove $\CCC\subset \BB^-$. Assume $\CCC(v,w)$, that is, there exists a chain of directed edges $v=v_0\rightarrow \dots\rightarrow v_n=w$. Suppose that $v_i\in F^-(w)$. Since $O^-$ is a stratification, it follows that $O^-(v_i)\subset F^-(w)$, and hence $v_{i-1}\in F^-(w)$. Thus, using decreasing induction on $i$, we can deduce that every $v_i\in F^-(w)$. Therefore, we have $\BB^-(v,w)$.
\end{proof}
\begin{example}\label{exm:nostrat}
    The converse of the previous lemma is not true. Let $P$ be a three dimensional polytope with five vertices $A,B,C,D,E$ as shown on the picture below.

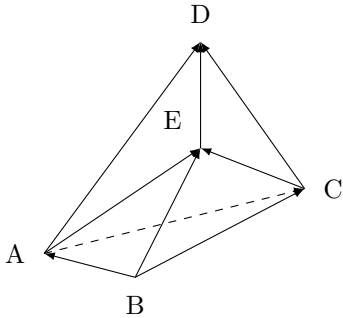
\begin{center}
\tdplotsetmaincoords{70}{30} 

\begin{tikzpicture}[tdplot_main_coords,>=latex]

\coordinate (A) at (-0.8, 0, 0);  
\coordinate (B) at (0.7, -0.2, 0);
\coordinate (C) at (1.6, 2.7, 0.5);
\coordinate (D) at (1, 1, 3);   
\coordinate (E) at (1, 1, 1.5); 

\draw[dashed,->] (A) -- (C); 

\draw[->] (B) -- (C); 
\draw[->] (B) -- (A); 
\draw[->] (A) -- (D); 
\draw[->] (C) -- (D); 
\draw[->] (A) -- (E); 
\draw[->] (B) -- (E); 
\draw[->] (C) -- (E); 
\draw[->] (E) -- (D); 

\node[label=left:A] at (A) {};
\node[label=below:B] at (B) {};
\node[label=right:C] at (C) {};
\node[label=above:D] at (D) {};
\node[label=above left:E] at (E) {};
\end{tikzpicture}
\end{center}
Here $\CCC$, $\BB^-$, $\BB^+$ and $\OO$ all coincide with the linear order $B,A,C,E,D$. 
However, neither $O^+$ nor $O^-$ is a stratification. For example, $C\in F^-(E)$, but $F^-(C)$ contains the face $ABC$, which is not contained in $F^-(E)$. Note that in this case $F^-(E)$ is reducible. We will later see that such examples are not possible when all $F^-(v)$ and $F^+(v)$ are irreducible for every vertex $v$.
\end{example}
\begin{lemma}\label{lem:face containing a chain}
    Assume $O^-$ is a stratification. Given any chain of directed edges
    \[
    v_0\rightarrow\dots\rightarrow v_k
    \]
    in $P$, there exists a face in $F^-(v_k)$ containing all vertices $v_0, \dots, v_k$.
\end{lemma}
\begin{proof}
    We argue by induction on $i$. The statement is trivial for $i=0$. Assume it holds for $i-1$, and $G\subset F^-(v_{i-1})$ is a face containing $v_0,\dots, v_{i-1}$. Since $O^-$ is a stratification, we have $F^{-}(v_{i-1}) \subset F^{-}(v_i)$. Hence, there exists a face $G' \subset F^{-}(v_i)$ containing $G$ such that $\ell$ attains its maximum on $G'$ at $v_i$. In particular, $v_i \in G'$, and hence $G'$ contains all vertices $v_0,\dots,v_i$, completing the inductive step.
\end{proof}

Given a face $F$ of $P$, a directed edge $v\to w$ in $F$ is called \emph{saturated} if there does not exist a chain of at least two directed edges from $v$ to $w$ in $F$. Moreover, we call a chain of directed edges \emph{saturated} in $F$ if each of its directed edge is saturated in $F$. Obviously, given a chain of directed edges in a face $F$ from $v$ to $w$, there exists a saturated chain of directed edges in $F$ from $v$ to $w$ which contains every vertex of the original chain.

The notion of saturated is relative to the face $F$. For example, the edge $C\to D$ in Example \ref{exm:nostrat} is saturated in the face $ACD$, but not in the face $CDE$. Nevertheless, if a chain is saturated in a face $F$ and belongs to a face $G\subset F$, then it is also saturated in $G$. 

We note that in Example \ref{exm:nostrat}, although the poset is graded, it is not graded with $\rho(v)=\dim F^-(v)$. This motivates the following lemma.
\begin{lemma}\label{lem:gradedposet}
    Assume $O^-$ is a stratification. Then $\BB^-=\CCC$ defines a graded poset with grading defined by $\rho(v)=\dim F^-(v)$. 
\end{lemma} 
\begin{proof}
    First we claim that $\dim F^-(v)$ strictly increases along directed edges. In fact, given a directed edge $v\rightarrow w$, the assumption that $O^-$ is a stratification implies 
    \[
    F^-(v)\subset F^-(w).
    \]
    Let $G$ be a face in $F^-(v)$ of maximal dimension. Then $G$ is contained in some face $G'$ in $F^-(w)$. Since $w\in G'$ but $w\not \in G$, we have $G\neq G'$, and hence $\dim G<\dim G'$. This proves the claim.

    We now show that all maximal saturated chains of directed edges have length $\dim P$. The proof proceeds by induction on $\dim P$. By Corollary \ref{cor:stratonface}, we may assume the statement holds for all proper faces of $P$. 
    Let 
    \[
    v_0\rightarrow\dots\rightarrow v_k
    \]
    be a saturated chain of directed edges, where $v_0$ is the source and $v_k$ is the sink of $P$. 


    By Lemma \ref{lem:face containing a chain}, there exists a face $G \subset F^{-}(v_{k-1})$ containing
    $v_0,\dots,v_{k-1}$. Since $G$ is a proper face of $P$, the inductive hypothesis implies that
    any saturated chain in $G$ has length $\dim G$. The chain
    \[
    v_0 \rightarrow \cdots \rightarrow v_{k-1}
    \]
    is saturated in $P$, and hence also saturated in $G$. Therefore, $\dim G=k-1$. In particular, any face of $F^{-}(v_{k-1})$ containing all vertices $v_0,\dots,v_{k-1}$ must have dimension at least $k-1$. Since $G$ already has this dimension, it follows that $G$ is
    inclusion-maximal among faces of $F^{-}(v_{k-1})$.

    We now show that $G$ is a facet of $P$. Suppose otherwise. Since every proper face is the intersection of all facets containing it and since $v_k\notin G$, $G$ is contained in a facet $G''$ of $P$ that does not contain $v_k$. By maximality of $G$ in
$F^{-}(v_{k-1})$, the function $\ell$ must attain its maximum on $G''$ at some vertex
$u \neq v_{k-1}$. By Lemma \ref{lem:path}, there exists a directed chain from $v_{k-1}$ to $u$, and another
directed chain from $u$ to $v_k$. This is a contradiction to $v_0\rightarrow\dots\rightarrow v_k$ being a saturated chain.
    \end{proof}
    \begin{corollary}\label{cor: containing chains=containing faces}
        If $O^-$ is a stratification, then each maximal saturated chain of directed edges  in a face $F$ has length $\dim F$. Further, if a face $G$ contains a maximal, saturated chain of edges in a face $F$, then $G$ contains $F$.
    \end{corollary}
\begin{proof}
    For the first part, by Corollary \ref{cor:stratonface}, we may apply Lemma \ref{lem:gradedposet} to the face $F$. For the second part, we consider the face $G\cap F$. It is contained in $F$, but has dimension equal to the length of the saturated path, which is also equal to $\dim F$. Thus $G\cap F=F$.
\end{proof}

    \begin{lemma}\label{lem:dimsum=dimP}
    For any vertex $v$ of $P$, we have 
    \[
    \dim F^+(v)+\dim F^-(v)\geq \dim P.
    \]
    Moreover, if $O^-$ is a stratification, then equality holds.
\end{lemma}
\begin{proof}
    We prove the first statement by induction on the dimension of the polytope. The equivalent statement is that given a pointed strongly convex polyhedral cone $C$ and a linear function $\ell$ which zero locus does not contain any rays, there exist two faces $F_1, F_2$ of $C$ on opposite sides of $\ell=0$ such that $\dim F_1+\dim F_2\geq \dim C$. 

    Let $R$ be a polytope obtained as a general cross section of $C$. In particular, the restriction of $\ell$ to $R$ is nonconstant along every edge. We have $\dim R=\dim C-1$. If $R$ is contained in one side of $\ell=0$, then the statement is obvious. So we assume that $R$ has vertices in both halfspaces $\{\ell>0\}$ and $\{\ell<0\}$. Let $w$ be a vertex of $R$ with minimal positive value of $\ell$ and let $R_w$ be the cone generated by $R-w$. 
    Let $\ell'$ be the linear function on $R_w$ defined by $\ell-\ell(w)$. By induction, we obtain a face $F_1'$ in $\{\ell'\geq 0\}$ and a face $F_2'$ in $\{\ell'\leq 0\}$ of $R_w$ such that $\dim F_1'+\dim F_2'\geq \dim R_w=\dim R$. The face $F_1'$ of $R_w$ corresponds to a face $F''_1$ of $R$ of the same dimension and to a face $F_1$ of $C$ of dimension one higher, the latter contained in the positive halfspace with respect to $\ell$. The face $F_2'$ corresponds to a face $F''_2$ of $C$ of dimension one higher where $\ell$ attains negative value on all rays, apart from one passing through $w$. Let $F_2$ be a facet of $F''_2$ not containing $w$. Note that $F_2$ is contained in the negative halfspace with respect to $\ell$. Thus, we have
    \[\dim F_1+\dim F_2=\dim F_1'+1+\dim F_2'+1-1\geq \dim R_w+1=\dim R+1=\dim C.\]

   When $O^-$ is a stratification we will apply Corollary \ref{cor: containing chains=containing faces}. Choose a saturated chain in a maximal dimensional face of $F^-(v)$ connecting the source and $v$. It has length $\dim F^-(v)$. Choose a saturated chain in the maximal dimensional face of $F^+(v)$ connecting $v$ and the sink. It has length $\dim F^+(v)$. Their union is a chain connecting the source to the sink of length $\dim F^-(v)+\dim F^+(v)$. We finish the proof by noting that such a chain can have length at most $\dim P$. 
\end{proof}

\begin{lemma}\label{lem:O=C}
    If $O^-$ is a stratification, then $\OO$ coincides with $\CCC$. 
\end{lemma}
\begin{proof}
    By Lemma \ref{lem:C=closO}, it is enough to prove that $\OO$ is transitive. Assume $\OO(v_1,v_2)$ and $\OO(v_2,v_3)$ and let $F_1$ and $F_2$ be faces of $P$ which witness these relations. As $O^-$ is a stratification, $F_1\subset F^-(v_3)$, and hence there exists a face $F\subset F^-(v_3)$ that contains $F_1$ and $v_3$. Replacing $P$ by $F$, we may assume that $v_3$ is the sink. Let $G$ be a maximal dimensional face in $F^+(v_1)$. 
    We will prove that $v_3\in G$, which finishes the proof.
    In fact, by Lemma \ref{lem:dimsum=dimP}, we have
    \[
    \dim F^-(v_1)+\dim G=\dim F^-(v_1)+\dim F^+(v_1)=\dim P.
    \]
    By Corollary \ref{cor: containing chains=containing faces} 
    connecting length-maximal chains of directed edges in $F^-(v_1)$ and in $G$, we have a chain of directed edges in $F^-(v_1)\cup G$ of length $\dim F^-(v_1)+\dim G=\dim P$. This must be a maximal saturated chain of directed edges in $P$. In particular, the end of the chain must be $v_3$, which shows that $v_3\in G$. 
\end{proof}
\begin{remark}
   Note that the converse of the previous corollary does not hold; see Example \ref{exm:nostrat}.
\end{remark}
As an immediate consequence of Lemmas \ref{lem:easyorderinclusions} and \ref{lem:O=C}, we have the following corollary.
\begin{corollary}\label{cor:gradthen graded}
    If $O^-$ is a stratification, then $\OO=\BB^-=\BB^+=\CCC$, and they define a graded poset with $\rho(v)=\dim F^-(v)$. 
\end{corollary}

\begin{lemma}\label{lem: F^- pure}
    If $O^-$ is a stratification, then all inclusion-maximal faces in $F^-(v)$ are of the same dimension.
\end{lemma}
\begin{proof}
    Choose an inclusion-maximal face $G$ in $F^-(v)$. Pick a maximal saturated chain 
    \[
    {s}=(v_0\to \cdots \to v_k=v)
    \]
    of directed edges in $G$. By Corollary \ref{cor: containing chains=containing faces}, $\dim G=k$ and any face that contains ${s}$ must contain $G$. 

    We can choose a saturated chain ${s}'$ of directed edges from the source to $v$, which also contains all the vertices in ${s}$. By Lemma \ref{lem:face containing a chain}, there exists a face $G'$ in $F^-(v)$ containing ${s}'$. Moreover, without loss of generality, we may assume that $G'$ is inclusion-maximal in $F^-(v)$. By the first part of Corollary~\ref{cor: containing chains=containing faces}, we have $\dim G'$ equal to the length of the saturated chain ${s}'$. By Corollary \ref{cor:gradthen graded} the length of this chain equals $\dim F^-(v)$. Thus:
    \[
    \dim G'=\dim F^-(v).
    \]
    By the second part of Corollary~\ref{cor: containing chains=containing faces}, $G'$ must contain $G$. Since $G$ is inclusion-maximal in $F^-(v)$, it follows that $G'=G$ and $\dim G=\dim F^-(v)$.
\end{proof}
\begin{corollary}\label{cor:max F^- have source}
    If $O^-$ is a stratification, then for any vertex $v$, any inclusion-maximal face $F$ in $F^-(v)$ contains the source.
\end{corollary}
\begin{proof}
    By Lemma \ref{lem: F^- pure}, we know that $\dim F=\dim F^-(v)$. By Lemma \ref{lem:gradedposet}, there exists a chain of directed edges inside $F$ of length $\dim F$. If $F$ did not contain the source, we could extend this chain, contradicting Lemma \ref{lem:gradedposet}.
\end{proof}
\begin{lemma}\label{lem: F^+ pure}
     If $O^-$ is a stratification, then all inclusion-maximal faces in $F^+(v)$ are of the same dimension.
\end{lemma}
\begin{proof}
    We proceed by the induction on the dimension of $P$. If $v$ is the source, then $F^+(v)=P$ and the statement is trivial. Now, assume the source does not belong to $F^+(v)$. Pick an inclusion-maximal face $G$ of $F^+(v)$. Let $H$ be a facet of $P$ that contains $G$, but does not contain the source. By Corollary \ref{cor:max F^- have source}, we have $\dim F^-_H(v)<\dim F^-_P(v)$. However, $G$ is still an inclusion-maximal face in $F^+_H(v)$. By induction, $\dim F^+_H(v)=\dim G$. Further, by Lemma \ref{lem:dimsum=dimP}, we have:
    \[\dim F^+_H(v)+\dim F^-_H(v)=\dim P-1. \]
    This implies:
    \[\dim G+\dim F^-_P(v)-1\geq \dim F^+_H(v)+\dim F^-_H(v)=\dim P-1. 
    \]
    Together with Lemma \ref{lem:dimsum=dimP}, we obtain:
    \[\dim G\geq \dim P-\dim F^-_P(v)=\dim F^+(v).\]
    Therefore, $\dim G=\dim F^+(v)$. 
\end{proof}
Using the previous lemma, in analogy to Corollary \ref{cor:max F^- have source}, we have the following corollary.
\begin{corollary}\label{cor:max F^+ have sink}
     If $O^-$ is a stratification, then for any vertex $v$ any inclusion-maximal face $F$ in $F^+(v)$ contains the sink.
\end{corollary}
\begin{theorem}\label{thm: O- strat iff O+}
    The subdivision $O^-$ is a stratification if and only if $O^+$ is a stratification.
\end{theorem}
\begin{proof}
 We proceed by induction on the dimension of $P$. Assume $O^-$ is a stratification. 

To prove $O^+$ is a stratification, choose vertices $p$ and $v$ such that $O^+(v)$ intersects $F^+(p)$, and we need to show that $O^+(v)\subset F^+(p)$. Notice that if the relative interior of a face is contained in $O^+(v)$, then this face contains $v$. Thus, $O^+(v)\cap F^+(p)\neq \emptyset$ implies that $v\in F^+(p)$. Let $W$ be any inclusion-maximal face of $F^+(v)$. It suffices to show that $W\subset F^+(p)$.



Pick an inclusion-maximal face $G$ in $F^+(p)$ that contains $v$. Take a saturated in $G$ chain $s$ from $p$ to $v$. Then take a saturated in $W$ chain $s'$ from $v$ to the sink of $W$, which, by Corollary \ref{cor:max F^+ have sink}, is also the sink of $P$. 
If $s'$ consists of a single edge, then so does  $W$. It has to be the edge from $v$ to the sink.
As $v$ and the sink belong to $G$ we have $W\subset G$. Hence, from now on we may assume that $s'$ has length at least two. 

 Let $s''$ be a chain that is the concatenation of $s$ and of $s'$ without the last edge. 
Assume $s''$ is of the form $p=v_0\rightarrow \dots\rightarrow v_n=w$. 
By Lemma~\ref{lem:face containing a chain}, there is a face $H$ in $F^-(w)$ that contains the chain $s''$. Since $w$ is the second last vertex of the chain $s'$, $w$ is not the sink of $P$, and hence $H$ is a proper face of $P$. In particular, applying inductive hypothesis to $H$, we know that $O^+_H$ is a stratification. Thus, by the $O^+$ analogue of Lemma~\ref{lem:face containing a chain}, there exists a face $H'$ in $F^+_H(p)$ containing the chain $s''$. 

Note that $H'\subset F^+_H(p) \subset F^+(p)$. Thus it must be contained in some inclusion-maximal face $H''$ in $F^+(p)$.
Note that $H''$ contains $H'$, and hence the chain $s''$. Moreover, by Corollary \ref{cor:max F^+ have sink}, $H''$ also contains the sink of $P$. Therefore, $H''$ contains every vertex of the chain $s'$, and hence the chain $s'$. 
Since $s'$ is a maximal saturated chain in $W$, by Corollary \ref{cor: containing chains=containing faces}, $H''$ contains $W$. Thus, $W\subset H''\subset F^+(p)$, which concludes the proof. 
\end{proof}

The following lemma is completely elementary and it is demonstrated by the figures below.
\begin{lemma}\label{lem:2-face}
\leavevmode
\begin{enumerate}
    \item For every triangle, $O^-$ is always a stratification. 
    \item For a convex quadrilateral, $O^-$ may or may not be a stratification. 
    \item For any convex $n$-gon with $n\geq 5$, $O^-$ is never a stratification.
\end{enumerate}
\noindent In particular, if a higher-dimensional polytope contains a face that is an $n$-gon with $n\geq 5$, then $O^-$ is not a stratification.
\end{lemma}
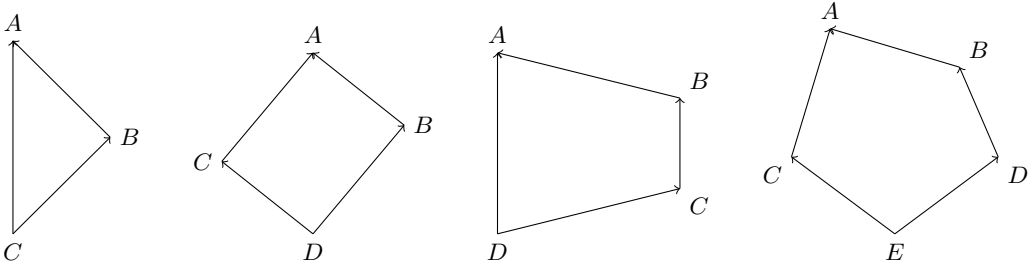
\begin{figure}[h]
\centering
\begin{tikzpicture}[scale=1.6, every node/.style={font=\small}]
    \coordinate (T1) at (0,0);
    \coordinate (T2) at (-0.8,-0.8);
    \coordinate (T3) at (-0.8,0.8);

    \draw[->] (T2) -- (T1);
    \draw[->] (T1) -- (T3);
    \draw[->] (T2) -- (T3);

    \node[above] at (T3) {$A$};
    \node[right] at (T1) {$B$};
    \node[below] at (T2) {$C$};
\end{tikzpicture}
\quad
\centering
\begin{tikzpicture}[scale=1.2, every node/.style={font=\small}]
    \coordinate (A) at (0,-0.2);
    \coordinate (B) at (1,-1);
    \coordinate (C) at (2,0.2);
    \coordinate (D) at (1,1);

    \draw[->] (B) -- (A);
    \draw[->] (B) -- (C);
    \draw[->] (C) -- (D);
    \draw[->] (A) -- (D);

    \node[left] at (A) {$C$};
    \node[below] at (B) {$D$};
    \node[right] at (C) {$B$};
    \node[above] at (D) {$A$};
\end{tikzpicture}
\quad
\centering
\begin{tikzpicture}[scale=1.2, every node/.style={font=\small}]
    \coordinate (A) at (0,1);
    \coordinate (B) at (0,-1);
    \coordinate (C) at (2,-0.5);
    \coordinate (D) at (2,0.5);

    \draw[->] (D) -- (A);
    \draw[->] (C) -- (D);
    \draw[->] (B) -- (C);
    \draw[->] (B) -- (A);

    \node[above] at (A) {$A$};
    \node[below] at (B) {$D$};
    \node[below right] at (C) {$C$};
    \node[above right] at (D) {$B$};
\end{tikzpicture}
\quad
\centering
\begin{tikzpicture}[scale=1.7, every node/.style={font=\small}]
    \coordinate (P1) at (0,0);
    \coordinate (P2) at (0.8,-0.6);
    \coordinate (P3) at (1.6,0);
    \coordinate (P4) at (1.3,0.7);
    \coordinate (P5) at (0.3,1);

    \draw[->] (P2) -- (P1);
    \draw[->] (P2) -- (P3);
    \draw[->] (P3) -- (P4);
    \draw[->] (P4) -- (P5);
    \draw[->] (P1) -- (P5);

    \node[below left] at (P1) {$C$};
    \node[below] at (P2) {$E$};
    \node[below right] at (P3) {$D$};
    \node[above right] at (P4) {$B$};
    \node[above] at (P5) {$A$};
\end{tikzpicture}

\caption{Examples of polygons with directed edges.}
\label{figure1}
\end{figure}

The previous lemma, combined with proposition below give many examples when $O^-$ (and similarly $O^+$)  are stratifications. 

\begin{proposition}\label{prop:construct examples}
\leavevmode
\begin{enumerate}
\item 
For $i=1, 2$, let $P_i\subset \R^{n_i}$ be a polytope and $\ell_i$ a linear function on $\R^{n_i}$ that is nonconstant on every edge of $P_i$. Assume that the induced partition $O_{P_i}^-$ is a stratification for both $i=1, 2$. Then the partition $O^-_{P_1\times P_2}$ of $P_1\times P_2$ induced by $\ell_1+\ell_2$ as a function on $\R^{n_1}\times \R^{n_2}$ is also a stratification. 
\item Let $P\subset \R^n$ be a polytope and $\ell$ is a linear function on $\R^n$ that is nonconstant on every edge of $P$. Let $\Con(P)\subset \R^{n+1}$ denote the pyramid over $P$, and we identify $P$ with the base of the pyramid $\Con(P)$. Let $\ell'$ be a linear function on $\R^{n+1}$ whose restriction to $P$ is equal to $\ell$ and whose unique minimum or maximum on $\Con(P)$ is the apex. Assume that the partition $O^-_P$ of $P$ induced by $\ell$ is a stratification. Then the induced partition $O_{\Con(P)}^-$ of $\Con(P)$ is also a stratification. 
\end{enumerate}
\end{proposition}
\begin{proof}
   For part (1), notice that the product of two stratifications is always a stratification. Thus, it is enough to show that 
    \[
    O^-_{P_1\times P_2}(v,w) =  O^-_{P_1}(v)\times O^-_{P_2}(w).
    \]
    This equality follows from the observation that for every face $G\times H$ of $P_1\times P_2$,  the function $(\ell_1,\ell_2)$ attains its maximum at vertex $(v,w)$ if and only if $\ell_1$ attains the maximum on $G$ at $v$ and $\ell_2$ attains the maximum on $H$ at $w$. 

    For part (2), denote the apex of the pyramid by $v_0$. If $v_0$ is the maximum of $\ell'$ on $\Con(P)$, then for any vertex $v$ of $P$, we have
    \[
    \overline{O_{\Con(P)}^-(v)}=\overline{O_{P}^-(v)}=\bigcup_{\OO_P(w, v)}O_{P}^-(w)=\bigcup_{\OO_{\Con(P)}(w, v)}O_{\Con(P)}^-(w).
    \]
    Additionally, $\overline{O_{\Con(P)}^-(v_0)}=\Con(P)$.
    If $v_0$ is the minimum of $\ell'$ on $\Con(P)$, then for any vertex $v$ of $P$, we have
    \[
    \overline{O_{\Con(P)}^-(v)}=\Con\left(\overline{O_{P}^-(v)}\right)=\{v_0\}\cup \bigcup_{\OO_P(w, v)}O_{\Con(P)}^-(w)=\bigcup_{\OO_{\Con(P)}(w, v)}O_{\Con(P)}^-(w).
    \]
    Additionally, $\overline{O_{\Con(P)}^-(v_0)}=\{v_0\}$. Therefore, $O_{\Con(P)}^-$ is always a stratification. 
\end{proof}
\begin{remark}
    In part (2) of the above proposition, the assumption $v_0$ being the source or sink is necessary. For example, in the first quadrilateral of Figure \ref{figure1}, if the apex satisfies $\ell'(B)>\ell'(v_0)>\ell'(C)$, then $O_{\Con(P)}^-$ is no longer a stratification. In fact, $O_{\Con(P)}^-(v_0)$ intersects the closure of $O_{\Con(P)}^-(B)$, but it is not contained in the closure. 
\end{remark}

The following theorem gives a classification of simple polytopes whose 2-faces are either triangles or quadrilaterals. 

\begin{proposition}[{\cite[Proposition 4.5]{Wiemeler}, see also \cite[Theorem 2.1]{YuMasuda}}]\label{prop:2faces=>simple}
    For a simple convex polytope $P$, all 2-faces of $P$ are either triangles or quadrilaterals if and only if $P$ is combinatorially equivalent to a product of simplices. 
\end{proposition}
Thus, by Lemma \ref{lem:2-face}, we immediately have the following corollary. 
\begin{corollary}\label{cor:prod_sim}
    If $P$ is a simple polytope satisfying Assumption \ref{ass:equivalent}, then $P$ is combinatorially equivalent to a product of simplices. 
\end{corollary}
\begin{remark}\label{rem:product of chains}
    In fact, when $P$ is a simple polytope satisfying Assumption \ref{ass:equivalent}, the poset $\CCC_P$ is isomorphic to a product of chains. As this fact is not needed in the remainder of the paper, we omit the proof.
\end{remark}

\subsection{Irreducible Bia{\l}ynicki-Birula strata}
In order to avoid the pathological behavior in Example~\ref{exm:nostrat}, we recall the  irreducibility assumption from the introduction.

\begin{assumption}\label{ass:irred}
    For every vertex $v$ of $P$, both $F^-(v)$ and $F^+(v)$ are irreducible, that is, each of them form a face of $P$.
\end{assumption}
\begin{lemma}
If Assumption \ref{ass:irred} holds for $P$, then it also holds for any face $G$ of $P$.
\end{lemma}
\begin{proof}
    By statement (2) of Lemma \ref{lem:F-intersect gives F-}, we have $F^-_G(w)=F^-_P(w)\cap G$. As intersection of two faces is a face, we see that $F^-_G(w)$ is indeed a face of $G$.
\end{proof}

\begin{lemma}
If $P$ is a simple polytope, then for any vertex $v$ of $P$ we have $F^-(v)$ is the face of $P$ spanned by edges incoming to $v$ and $F^+(v)$ is the face spanned by the edges outgoing from $v$. In particular, Assumption~\ref{ass:irred} always holds. 
\end{lemma}
\begin{proof}
    By symmetry we only prove the statement for $F^-(v)$. Let $G$ be the face spanned by all edges incoming to $v$. Then $v$ is a sink for $G$ and thus $G\subseteq F^-(v)$. 

    Let $F$ be a face in $F^-(v)$. Then $v$ is the sink on $F$, i.e.,~edges of $F$ adjacent to $v$ are all incoming edges. As every face is the face spanned by edges adjacent to any of its vertices,
    $F$ is spanned by some subset of edges incoming to $v$. In particular, $F\subseteq G$ and thus $F^-(v)\subseteq G$.
    \end{proof}

\begin{definition}
Given a face $G$, a directed edge $v_1\to v_2$ is called an \emph{incoming edge} (resp.~\emph{outgoing edge}) of $G$, if $v_2\in G$ but $v_1\notin G$ (resp. if $v_1\in G$ but $v_2\notin G$). 
\end{definition}

The main result of this subsection is the following theorem.
\begin{theorem}\label{thm:equivalences}
    Assume that $\ell$ is a linear function on a polytope $P\subset \RR^n$ that is not constant on edges. If Assumption \ref{ass:irred} holds, then the following conditions are equivalent:
    \begin{enumerate}
        \item $O^-$ is a stratification,
        \item $O^+$ is a stratification,
        \item for every vertex $v$ the face $F^-(v)$ has no incoming edges,
        \item for every vertex $v$ the face $F^+(v)$ has no outgoing edges,
        \item $\OO=\CCC$,
        \item $\OO$ is a poset,
        \item $\OO$ is a graded poset, with grading given by $\rho(v)=\dim F^-(v)$,
        \item for any vertices $v,w$, we have 
        \[
        \overline{O^+(v)\cap O^-(w)}=F^+(v)\cap F^-(w).
        \]
    \end{enumerate}
\end{theorem}
We start with a lemma that proves equivalences $(1)\Leftrightarrow (3)$ and $(2)\Leftrightarrow (4)$ in Theorem \ref{thm:equivalences}. 
\begin{lemma}\label{lem:incomingEdges}
    Under Assumption \ref{ass:irred}, the partition $O^-$ is a stratification if and only if, for any vertex $v$, the face $F^-(v)$ has no incoming edges.  Analogously, $O^+$ is a stratification if and only if, for every vertex $v$, the face $F^+(v)$ has no outgoing edges. 
\end{lemma}
\begin{proof}
We only prove the first statement, as the second is similar.
    Assume $O^-$ is a stratification.  Consider a directed edge $v_1\to v_2$ where $v_2\in F^-(v)$. We want to prove that $v_1\in F^-(v)$. The intersection $O^-(v_2)\cap F^-(v)$ contains $v_2$ and hence is nonempty. Thus $O^-(v_2)\subseteq F^-(v)$ by assumption on the stratification. Since $v_1$ is in the closure of $O^-(v_2)$ and $F^-(v)$ is closed, we conclude that $v_1\in F^-(v)$.

    Conversely, assume that for any vertex $v$, the face $F^-(v)$ has no incoming edges. Suppose that the intersection $O^-(v_1)$ and $F^-(v_2)$ contains the relative interior of a face $F_3$. Note that $F_3$ is a face of $F^-(v_2)$. Further, $\ell$ attains its maximum on $F_3$ at $v_1$ and in particular $F_3$ contains $v_1$. Thus $v_1\in F^-(v_2)$.
    
    Finally we prove that each vertex $w$ of $F^-(v_1)$ is contained in $F^-(v_2)$. Indeed, by Lemma \ref{lem:path} we have a sequence of directed edges connecting $w$ and $v_1\in F^-(v_2)$. Inductively, each such edge must be contained in $F^-(v_2)$ and thus so must $w$. As each vertex of $F^-(v_1)$ is contained in $F^-(v_2)$, by Assumption \ref{ass:irred}, we have $F^-(v_1)\subseteq F^-(v_2)$.  
\end{proof}

We are now ready to finish the proof of the main theorem of this subsection. 

\begin{proof}[Proof of Theorem \ref{thm:equivalences}]
    The equivalences $(1)\Leftrightarrow (3)$ and $(2)\Leftrightarrow (4)$ follow from Lemma \ref{lem:incomingEdges}. The equivalence $(1)\Leftrightarrow (2)$ is Theorem \ref{thm: O- strat iff O+}. Thus, the first four conditions are equivalent. 

    The implication $(6)\Rightarrow (5)$ follows from Lemma \ref{lem:C=closO}. The implication $(5)\Rightarrow (6)$ is trivial as $\CCC$ is always a poset.
The implication $(1)\Rightarrow (7)$ is Corollary \ref{cor:gradthen graded}. The implication $(7)\Rightarrow (6)$ is obvious. 

    Next we prove $(6)\Rightarrow (1)$. If $\OO$ is a poset, then by Lemma \ref{lem:C=closO}, we know that $\OO=\CCC$. Furthermore, by Lemma~\ref{lem:easyorderinclusions}, we know that $\BB^-=\OO$ is a poset. Therefore, for any vertex $v$, the face $F^-(v)$ does not have incoming edges. Thus by Lemma \ref{lem:incomingEdges}, condition $(1)$ follows.

    So far, we proved the equivalence of the first seven conditions.

    We now prove that $(1)\Rightarrow (8)$. Since ${O^+(v)\cap O^-(w)}\subseteq F^+(v)\cap F^-(w)$, we have
    \[
    \overline{O^+(v)\cap O^-(w)}\subseteq F^+(v)\cap F^-(w).
    \]
    Let $G=F^+(v)\cap F^-(w)$. Then, the statement is trivial when $G=\emptyset$. If $G\neq \emptyset$, by Corollary \ref{cor:gradthen graded}, $w\in F^+(v)$ and $v\in F^-(w)$. Therefore, in $G$, $\ell$ attains minimum at $v$ and maximum at $w$. 
    Thus the interior of $G$ belongs to $O^+(v)\cap O^-(w)$, and hence, $G\subset \overline{O^+(v)\cap O^-(w)}$. Condition $(8)$ follows.

    Finally, we prove $(8)\Rightarrow (3)$. 
    For any vertex $v$ and any directed edge $v_1\to v_2$ with $v_2\in F^-(v)$. We have 
    \[
    v_2\in F^+(v_1)\cap F^-(v)=\overline{O^+(v_1)\cap O^-(v)}.
    \]
    In particular,  $O^+(v_1)\cap O^-(v)$ is nonempty, i.e.,~there exists a face $G$ on which $\ell$ attains maximum at $v$ and minimum at $v_1$. This proves $v_1\in F^-(v)$, and hence, $F^-(v)$ has no incoming edges.
\end{proof}

\begin{example}\label{exm:triandquad}
    We note that the equivalent conditions in Theorem \ref{thm:equivalences} do not depend only on the polytope $P$, but also on the function $\ell$. In dimension two, for any triangle, the equivalent conditions hold for general $\ell$. 
    
    For quadrilaterals, we observe the dependence from the two examples in Figure \ref{figure1}. 
    In particular, the equivalent conditions hold for a quadrilateral if and only if there is no edge between the source and the sink.






For $n\geq 5$, any $n$-gon does not satisfy the conditions in Theorem \ref{thm:equivalences} for any $\ell$.
\end{example}

Similarly to Proposition \ref{prop:construct examples}, we have the following constructions that allow to obtain many examples of polytopes satisfying both Assumption \ref{ass:irred} and the conditions in Theorem \ref{thm:equivalences}. 
\begin{proposition}\label{prop:constructexamples}
\begin{enumerate}
\item     If $(P_1,\ell_1)$ and $(P_2,\ell_2)$ satisfy both Assumption \ref{ass:irred} and the equivalent conditions in Theorem \ref{thm:equivalences}, then so does $(P_1\times P_2, \ell_1\times \ell_2)$.
\item If $P\subset \RR^n$ and $\ell: \RR^n\to \RR$ satisfy both Assumption \ref{ass:irred} and the equivalent conditions in Theorem \ref{thm:equivalences}, then so does the polytope in $\RR^{n+1}$ that is a pyramid over $P$ with new vertex $v$ that is the new source or sink. 
\item If $P$ satisfies Assumption \ref{ass:irred}, then any face of $P$ also satisfies Assumption \ref{ass:irred}. Moreover, if $P$ also satisfies the equivalent conditions in Theorem \ref{thm:equivalences}, then so does any face of $P$. 
    \end{enumerate}
\end{proposition}
\begin{proof}
    The statements (1) and (2) are straightforward. For (3), assume $P$ satisfies  Assumption \ref{ass:irred}. Let $G$ be a face of $P$ and $v$ be a vertex of $G$. Then both $F^-_G(v)=F^-_P(v)\cap G$ and $F^+_G(v)=F^+_P(v)\cap G$ are irreducible. Hence the first assertion follows. 
    The second part follows immediately from Corollary~\ref{cor:stratonface}. 
\end{proof}
\begin{remark}
    We note that Theorem \ref{thm:equivalences} is false without Assumption \ref{ass:irred}. Indeed, for Example \ref{exm:nostrat} conditions (3), (4), (5), (6) hold, while (1), (2), (7) and (8) do not. 
\end{remark}
\begin{remark}\label{rem:O is subposet of face poset}
    If Assumption~\ref{ass:irred} is satisfied, we may associate to every vertex $v$ the face $F^-(v)$. If additionally $\mathcal{O}$ is a poset, then it is the restriction of the face poset with inclusion, via $v\mapsto F^-(v)$. Moreover, this restriction preserves the grading.
\end{remark}

We conclude this section with the following lemma that will be used later.

\begin{lemma}\label{lem:dimensionsbehavenicely}
Assume that $P$ satisfies Assumption \ref{ass:irred} and $O^-$ is a stratification. For any vertices $u, w$ with $\OO(u,w)$. Then,
\[
\dim F^-(u)+\dim F^+(u)\cap F^-(w)+\dim F^+(w)=\dim P.
\]
\end{lemma}
\begin{proof}
    Choose a maximal saturated chain $v_0\to v_1\to \cdots \to v_k$ which contains $u$ and $w$.
Assume that $u=v_i$ and $w=v_j$. By Corollary \ref{cor:gradthen graded}, the chain $v_0\to \cdots \to v_i$ is contained in $F^-(u)$, the chain $v_i\to \cdots \to v_j$ is contained in $F^+(u)\cap F^-(w)$, and the chain $v_j\to \cdots \to v_k$ is contained in $F^+(w)$. Since $v_0\to v_1\to \cdots v_k$ is a saturated chain in $P$, $v_0\to \cdots \to v_i$ is maximal saturated chain in $F^-(u)$. Hence, by Corollary \ref{cor: containing chains=containing faces},
    \[
    \dim F^-(u)=i,
    \]
    and similarly, we have
    \[
    \dim F^+(u)\cap F^-(w)=j-i,\quad \dim F^+(w)=k-j. 
    \]
    Therefore, 
    \[
    \dim F^-(u)+\dim F^+(u)\cap F^-(w)+\dim F^+(w)=i+(j-i)+(k-j)=k=\dim P. \qedhere
    \]
\end{proof}

\section{Chow quotients of polytopes}\label{sec:Chowpolytope}

\subsection{Fiber polytopes and Chow quotients}
In their seminal papers, Billera and Sturmfels \cite{billera1992fiber} and  Kapranov, Sturmfels and Zelevinsky \cite{kapranov1991quotients} developed a combinatorial framework for taking quotients of toric varieties by (possibly low dimensional) tori. In our setting, the key point is that the Chow quotient of a toric variety by a torus action is again a toric variety, whose associated polytope can be described as a fiber polytope. In case of action of one-dimensional torus $\C^*$, fiber polytopes are also called \emph{monotone path polytopes} and described in detail in \cite[Section 5]{billera1992fiber}. 

Given a positive-dimensional polytope $P$ and a linear function $\ell$ that is nonconstant on $P$, the monotone path polytope is defined as the Minkowski sum of explicitly described cuts of $P$ given by $\ell=c$, as $c$ varies. We will denote this monotone path polytope by $\CH_\ell(P)$ or simply $\CH(P)$. When $P$ is a very ample lattice polytope, $\CH(P)$ represents the normalization of the Chow quotient of the toric variety associated with $P$ by the $\C^*$-action determined by $\ell$. When $P$ is positive-dimensional, the dimension of the monotone path polytope $\CH_\ell(P)$ is always $\dim P-1$. When $P$ is a point, we define $\CH(P)$ to be a point as well. 

We summarize the main results of \cite{billera1992fiber} as the following two theorems. The first theorem describes the set of faces of $\CH(P)$. 


\begin{theorem}\label{thm:facesCH}
Given a linear function $\ell$, faces of the monotone path polytope $\CH(P)$ are in bijection with sequences of faces $F_1,\dots,F_k$ of $P$ such that:
\begin{enumerate}
    \item the maximal vertex of $F_i$ (with respect to $\ell$) is the minimal vertex of $F_{i+1}$,
    \item the minimal vertex of $F_1$ is the source in $P$,
    \item the maximal vertex of $F_k$ is the sink in $P$,
    \item there exists a linear function $\mu$ so that for every cut $P\cap \{x:\ell(x)=c\}$, the function $\mu$ is minimized exactly on $F_i\cap \{x:\ell(x)=c\}$ for some $1\leq i\leq k$.
\end{enumerate}
Moreover, the face of $\CH(P)$ corresponding to $(F_1,\dots, F_k)$ is normally equivalent to
the Minkowski sum of monotone path polytopes $\CH(F_i)$. 
\end{theorem}

In particular, we get that the face of $\CH(P)$ corresponding to $(F_1,\dots, F_k)$ is combinatorially equivalent to the Minkowski sum of monotone path polytopes $\CH(F_i)$. 

Every linear function $\mu$ defines a face of $\CH(P)$ in the following way. For every slice $P\cap \{x:\ell(x)=c\}$ denote by $P_c^\mu$ the face of the slice which minimizes $\mu$. Then one can see that the union $\bigcup_c P_c^\mu$ coincides with the union of a sequence of faces $(F_1,\dots, F_k)$ of $P$ satisfying conditions of Theorem~\ref{thm:facesCH}, which defines a face of $\CH(P)$. 
The second theorem describes the poset structure of the faces of $\CH(P)$. 
\begin{theorem}
\label{thm:facelatticeCH}
   Two faces in $\CH(P)$ corresponding to sequences $(F_1, \ldots, F_k)$ and $(G_1, \ldots, G_j)$ have non-empty intersection if and only if
\[
(F_1\cup \cdots \cup F_k)\cap (G_1\cup \cdots \cup G_j)
\]
is the union of a chain of faces $(H_1,\ldots,H_s)$ satisfying condition (1) in Theorem~\ref{thm:facesCH}.
In that case, the intersection is a face of $\CH(P)$ given by the sequence $(H_1,\ldots,H_s)$.

In particular, the face of $\CH(P)$ corresponding to $(F_1, \ldots, F_k)$ contains the face corresponding to $(G_1, \ldots, G_j)$ if and only if
\[
G_1\cup \cdots \cup G_j\subset F_1\cup \cdots \cup F_k.
\]
\end{theorem}
 
When Assumption \ref{ass:irred} and the equivalent conditions in Theorem \ref{thm:equivalences} hold, the faces of the monotone path polytopes have easier characterizations. We recall the stratification assumption from the introduction. 

\begin{assumption}\label{ass:equivalent}
    In addition to Assumption \ref{ass:irred}, we assume that the equivalent conditions of Theorem~\ref{thm:equivalences} hold. 
\end{assumption}
\begin{remark}
   Throughout the rest of the paper, whenever we refer to this assumption, we will suppress the dependence on the linear function $\ell$ and simply say that $P$ satisfies Assumption~\ref{ass:equivalent}.  
\end{remark}

For the remainder of this section, we assume that Assumption~\ref{ass:equivalent} holds, and we will explicitly restate it only in the statements of theorems.
The main result of this section is Theorem~\ref{thm:simpleCH}, which characterizes those polytopes $P$ satisfying Assumption~\ref{ass:equivalent} for which the monotone path polytope $\CH(P)$ is simple. A key intermediate result is Theorem~\ref{thm:facesofCH}, which shows that, in contrast to the general situation, when $P$ satisfies Assumption~\ref{ass:equivalent}, the faces of $\CH(P)$ are described by chains of faces of $P$ only satisfying conditions (1),(2), and (3) of Theorem~\ref{thm:facesCH}.

\begin{lemma}\label{lem:dimatmostP}
    For any sequence of faces $(F_1,\dots, F_k)$ satisfying the conditions in Theorem \ref{thm:facesCH}, we have 
    \[
    \dim F_1+\cdots+\dim F_k\leq \dim P.
    \]
\end{lemma}
\begin{proof}
    We prove by induction on $k$. When $k=1$, the inequality is trivial. Let $v$ be the vertex at which $\ell$ attains maximum on $F_1$. We have $F_1\subseteq F^-(v)$ as well as $F_2,\dots, F_k\subseteq F^+(v)$. 
    By inductive hypothesis, we have
    \[
    \dim F_2+\cdots+\dim F_k\leq \dim F^+(v).
    \]
    Thus,
    \[
    \dim F_1+\cdots+\dim F_k\leq \dim F^-(v)+\dim F^+(v)=\dim P,
    \]
    where the last equality follows from Lemma~\ref{lem:dimsum=dimP}. 
\end{proof}

Our next goal is to obtain a more detailed description of the facets of $\CH(P)$. 
The following lemma shows that each facet is of one of two possible types. 
In Corollary~\ref{cor:facetsofCH}, we will show that the classification given in this lemma is bijective.

\begin{lemma}\label{lem:facetsofCHPtwotypes}
    Every facet of $\CH(P)$ is of one of the following types:
    \begin{enumerate}
        \item those corresponding to pairs $(F_1,F_2)$, where $F_1 = F^-(v)$ and $F_2 = F^+(v)$ for some vertex $v$ of $P$ distinct from the source and the sink;
        \item those corresponding to a single face $(F_1)$, where $F_1$ is a facet of $P$ containing both the source and the sink.
    \end{enumerate}
    In the first case, the facet is combinatorially equivalent to $\CH(F_1)\times \CH(F_2)$, while in the second case it is combinatorially equivalent to $\CH(F_1)$.
\end{lemma}
\begin{proof}
    By Theorem \ref{thm:facesCH}, every facet must be represented by a sequence $(F_1,\dots,F_k)$, where $F_1$ contains the source and $F_k$ contains the sink. 
    Since the facet is combinatorially equivalent to $\CH(F_1)+\cdots +\CH(F_k)$, its dimension satisfies
    \begin{multline*}
    \dim P-2=\dim \CH(P)-1=\dim\big(\CH(F_1)+\cdots +\CH(F_k)\big)\\
    \leq \dim \CH(F_1)+\cdots +\dim \CH(F_k)=\dim F_1+\cdots+\dim F_k-k\leq \dim P-k
    \end{multline*}
    where the last inequality follows from Lemma \ref{lem:dimatmostP}. Thus, we must have $k=1$ or 2. 
    
    When $k=1$, we have $\dim \CH(F_1)=\dim \CH(P)-1$, which implies that $\dim F_1=\dim P-1$. 
    
    When $k=2$, both inequalities in the above formula must be equalities. Thus, 
    \[
    \dim F_1+\dim F_2=\dim P.
    \]
    On the other hand, by Lemma \ref{lem:dimsum=dimP},
    \[
    \dim F^-(v)+\dim F^+(v)=\dim P
    \]
    where $v=F_1\cap F_2$. Since $F_1\subset F^-(v)$ and $F_2\subset F^+(v)$, the above two equalities imply that $F_1= F^-(v)$ and $F_2= F^+(v)$. 
    Finally, we get
    \[
    \dim \CH(F^-(v))+\dim \CH(F^+(v))=\dim F^-(v)+\dim F^+(v)-2=\dim P-2=\dim \CH(P)-1. 
    \]
    Therefore, 
\[ \dim\big(\CH(F^-(v))+\CH(F^+(v))\big)
=
\dim \CH(F^-(v))+\dim \CH(F^+(v)),\]
    which implies that
    \[
    \CH(F^-(v))+\CH(F^+(v))\cong \CH(F^-(v))\times \CH(F^+(v)). \qedhere
    \]

\end{proof}
\begin{lemma}\label{lem:facetdeg1}
    Let $v\in P$ be a vertex such that $\dim F^-(v)=\dim P-1$. Then $F^+(v)$ is equal to the directed edge $v\to w$, where $w$ is the sink of $P$. Moreover, the sequence $(F^-(v), F^+(v))=(F^-(v), v\to w)$ defines a facet of $\CH(P)$.
\end{lemma}

\begin{proof}
The first claim follows from Lemma \ref{lem:dimsum=dimP}. 

    Let $\mu$ be a linear function such that its minimum value on $P$ is achieved exactly on the face $F^-(v)$.
    Then $\mu$ defines a face of $\CH(P)$ represented by a sequence $(F^-(v),F_2,\dots,F_k) $ for some $F_2,\ldots,F_k$. However, each $F_i$ for $i=2,\dots, k$ is contained in $F^+(v)$, which is one-dimensional and hence the above sequence is $(F^-(v), v\to w)$. 
   Moreover, this face of $\CH(P)$ is combinatorially equivalent to 
    \[
    \CH(F^-(v))+\CH(v\to w)=\CH(F^-(v))
    \]
    which has dimension 
    \[
    \dim F^-(v)-1=\dim P-2=\dim \CH(P)-1.
    \]
    Hence, it is a facet of $\CH(P)$. 
\end{proof}


Our next goal is to give a direct description of the vertices of $\CH(P)$. 
By Theorem~\ref{thm:facesCH}, each vertex is represented by a sequence of directed edges starting at the source and ending at the sink. 
In general, such a path may not determines a vertex of $\CH(P)$, due to condition~(4) of Theorem~\ref{thm:facesCH}. 
However, under Assumption~\ref{ass:equivalent}, the vertices of $\CH(P)$ are in bijection with monotone paths. 
We first establish this correspondence for paths of maximal length.


\begin{lemma}\label{lem:F-F+isaFacet}
    For any vertex $v\in P$, there exists a facet of $\CH(P)$ represented by $(F^-(v), F^+(v))$.
\end{lemma}
\begin{proof}
    Consider a maximal saturated chain $v_0\to v_1\to \dots\to v_d$ that passes through $v$. We will prove by induction on $i$ that each $(F^-(v_{d-i}),F^+(v_{d-i}))$ corresponds to a facet of $\CH(P)$. The case $i=1$ follows from Lemma \ref{lem:facetdeg1}.

    Assume that $(F^-(v_{d-i}),F^+(v_{d-i}))$ represents a facet, which is combinatorially equivalent to 
    \[\CH(F^-(v_{d-i}))+\CH(F^+(v_{d-i}))\simeq \CH(F^-(v_{d-i}))\times\CH(F^+(v_{d-i})).\]
    Within this facet there is a codimension two face $G$ of $\CH(P)$, which is combinatorially equivalent to
\[
\CH(F^-(v_{d-i-1}))\times \CH(v_{d-i-1}\to v_{d-i})\times \CH(F^+(v_{d-i})).
\]
In other words, $G$ is determined by the sequence $(F^-(v_{d-i-1}), v_{d-i-1}\to v_{d-i}, F^+(v_{d-i}))$. 

As a  codimension two face of $\CH(P)$, $G$ must be the intersection of two facets of $\CH(P)$, one of which is $(F^-(v_{d-i}),F^+(v_{d-i}))$. 
By Theorem \ref{thm:facelatticeCH}, the other facet containing $G$ is not of type (2) in Lemma \ref{lem:facetsofCHPtwotypes}. In fact, any face of $P$ containing both the source $v_0$ and the sink $v_d$ contains the line segment $\mathrm{Conv}(v_0, v_d)$ (not necessarily an edge of $P$), but 
\[
F^-(v_{d-i-1})\cup  (v_{d-i-1}\to v_{d-i})\cup  F^+(v_{d-i})
\]
does not contain the line segment $\mathrm{Conv}(v_0, v_d)$. 

Thus, the other facet containing $G$ must be determined by $(F^-(w),F^+(w))$ for some vertex $w$ of $P$. By Theorem \ref{thm:facelatticeCH}, the vertex $w$ satisfies
\[
F^-(v_{d-i-1})\cup  (v_{d-i-1}\to v_{d-i})\cup  F^+(v_{d-i})\subset F^-(w)\cap F^+(w). 
\]
Restricting the above inclusion to the level set $\ell=\ell(w)$, it follows that $w=v_{d-i-1}$ or $v_{d-i}$. Thus, the other facet containing $G$ is determined by $(F^-(v_{d-i-1}), F^+(v_{d-i-1}))$. In particular, the two-term sequence $(F^-(v_{d-i-1}), F^+(v_{d-i-1}))$ corresponds to a facet of $\CH(P)$. 
\end{proof}
\begin{corollary}\label{cor:facetsofCH}
    The facets of $\CH(P)$ are in bijection with 
    \begin{enumerate}
        \item single-term sequences $(F)$ where $F$ is a facet of $P$ containing both source and sink, and
        \item two-term sequences $(F^-(v), F^+(v))$ where $v$ is a vertex of $P$ that is neither the source nor the sink. 
    \end{enumerate}
\end{corollary}
\begin{proof}
    By Lemma \ref{lem:facetsofCHPtwotypes}, we only need to show that such sequences always give rise to facets of $\CH(P)$. For the two-term sequences $(F^-(v), F^+(v))$, it follows from Lemma \ref{lem:F-F+isaFacet}. For the one-term sequences $(F)$, let $\mu$ be the linear function supporting $F$. Then $\mu$ justifies the last condition of Theorem \ref{thm:facesCH}. Hence, $(F)$ corresponds to a facet of $\CH(P)$.
\end{proof}
More generally, we prove that under Assumption \ref{ass:equivalent}, the last condition of Theorem \ref{thm:facesCH} is not needed.
\begin{theorem}\label{thm:facesofCH}
    Assume that $P$ satisfies Assumption \ref{ass:equivalent}. The faces of $\CH(P)$ are in bijection with sequences $(F_1,\dots, F_k)$ of faces of $P$, where the maximal vertex of $F_{i-1}$ is the minimal vertex of $F_i$ for $i=2,\dots, k$, the minimal vertex of $F_1$ is the source of $P$ and the maximal vertex of $F_k$ is the sink of $P$. The dimension of such a face of $\CH(P)$ equals 
    \[
    \dim F_1+\cdots +\dim F_k-k.
    \]
\end{theorem}
\begin{proof}
     It is enough to prove that every such chain corresponds to a face of the given dimension.
     We proceed by induction on $k$. When $k=1$, we need the linear function $\mu$ satisfies condition (4) of Theorem \ref{thm:facesCH}.   For this we can choose $\mu$ such that the minimum value of $\mu$ on $P$ is achieved on the face $F_1$. 

    Now, suppose $k>1$. Let $v$ be the vertex of $F_1$ on which $\ell$ attains maximum. By inductive hypothesis, the single-term sequence $(F_1)$ determines a face of $\CH(F^-(v))$, and the sequence $(F_2, \dots, F_k)$ determines a face of $\CH(F^+(v))$. By Lemma~\ref{lem:F-F+isaFacet}, the sequence $(F^-(v), F^+(v))$ determines a facet $G$ of $\CH(P)$, which by Lemma~\ref{lem:facetsofCHPtwotypes} is combinatorially equivalent to $\CH(F^-(v))\times \CH(F^+(v))$.
    
    The face of $\CH(F^-(v))$ determined by $(F_1)$ together with the face of $\CH(F^+(v))$ determined by $(F_2, \dots, F_k)$ gives rise to a face of $G$, and hence a face of $\CH(P)$. Obviously, such face of $\CH(P)$ is determined by the sequence $(F_1, \dots, F_k)$.    
%
\end{proof}

\begin{definition}
    We call a sequence of faces $(F_1, \ldots, F_k)$ as in Theorem \ref{thm:facesofCH} a \emph{monotone chain of faces} in $P$. We denote the corresponding face of $\CH(P)$ by $[F_1, \ldots, F_k]$. When all $F_i$ are edges, $(F_1, \ldots, F_k)$ is a maximal chain of directed edges $v_0\to \cdots \to v_k$, which we also call \emph{a monotone path}. In this case, we also denote the corresponding vertex in $\CH(P)$ by $[v_0\to \cdots \to v_k]$. 
\end{definition}

As a direct corollary, we obtain the following description of vertices and edges of $\CH(P)$.
\begin{corollary}\label{cor:vertofCHP}
    Vertices of $\CH(P)$ are in bijection with monotone paths (not necessarily saturated) in $P$ going from the source to the sink.
\end{corollary}

\begin{corollary}\label{cor:edgesofCHP}
Every edge of $\CH(P)$ corresponds to a monotone chain of faces $(F_1,\dots, F_k)$, where all but one of the $F_i$ are edges and the remaining face has dimension two. The two vertices of such an edge of $\CH(P)$ correspond to monotone paths in $F_1\cup \cdots \cup F_k$ that differ only within the two-dimensional face.
\end{corollary}

We also obtain a concrete description of the edges adjacent to a given vertex of $\CH(P)$. The following corollary is a straightforward consequence of Corollary \ref{cor:edgesofCHP} and Lemma \ref{lem:2-face}. 
\begin{corollary}\label{cor:edgesadjacenttovertexofCHP}
  Consider a vertex of $\CH(P)$ corresponding to a monotone path $v_0\to \dots\to v_k$. 
  Each edge of $\CH(P)$  adjacent to the given vertex corresponds to:
  \begin{enumerate}
      \item a triangle in $P$ with vertices $\{v_i,x, v_{i+1}\}$, where $v_i\to x$ and $x\to v_{i+1}$ are directed edges,
      \item a 2-face containing vertices $v_i, v_{i+1}, v_{i+2}$, and contained in $F^+(v_i)\cap F^-(v_{i+2})$.
  \end{enumerate}
  Moreover, the 2-face in (2) is either the triangle $(v_i, v_{i+1}, v_{i+2})$ or a quadrilateral with vertices $\{v_i, v_{i+1}, v_{i+2}, x\}$ such that $v_i\to x$ and $x\to v_{i+2}$ are directed edges.
\end{corollary}

\begin{corollary}\label{cor:numberoftrangles}
    Let $v$ and $w$ be the source and sink of $P$, respectively. If there is a directed edge $v\to w$, then the number of triangle faces $(v, w, x)$ of $P$ is at least $\dim P-1$. 
\end{corollary}
\begin{proof}
    By Corollary \ref{cor:vertofCHP}, the edge $v\to w$ corresponds to a vertex of $\CH_\ell(P)$, which we denote by $v'$. By Corollary~\ref{cor:edgesadjacenttovertexofCHP}, all edges of $\CH_\ell(P)$ connecting to $v'$ correspond to triangle faces $(v, w, x)$ of $P$. Since $\dim \CH_\ell(P)=\dim P-1$, the number of such triangle faces must be at least $\dim P-1$.
\end{proof}

\begin{remark}
   For an arbitrary polytope $P$ of dimension $n$ with a generic linear function $\ell$ its graph of $\ell$-monotone paths $G(P,\ell)$ is a graph defined via
   \begin{itemize}
       \item Vertices of  $G(P,\ell)$ are monotone paths connecting the source to the sink.
       \item Edges are chains of faces $(F_1,\dots, F_k)$, where all but one of the $F_i$ are edges and the remaining face has dimension two.
   \end{itemize}
  By Theorems~\ref{thm:facesCH} and~\ref{thm:facelatticeCH} and some mild assumptions of $P$ and $\ell$, we get that $G(P,\ell)$ is a full subgraph of the edge graph of monotone path polytope $\CH(P)$, which is $(n-1)$-connected by Balinski's theorem. Motivated by this observation Reiner conjectured that the graph $G(P,\ell)$ is always $(n-1)$-connected \cite[Conjecture 15]{reiner1999generalized}. Quickly after, in \cite{athanasiadis2000monotone} Reiner's conjecture was proved for simple polytopes $P$ but disproved in general.

  Our Theorem~\ref{thm:facesofCH} and Corollaries~\ref{cor:vertofCHP} and~\ref{cor:edgesofCHP} guarantee that for pairs $(P,\ell)$ satisfying Assumption~\ref{ass:equivalent} the graph of monotone paths coincides with the edge graph of Chow quotients, thus provides another class of polytopes for which Reiner's conjecture holds.
\end{remark}

\begin{lemma}\label{lem:bendface}
    Let $F_1$ and $F_2$ be two faces of $P$ such that the maximal vertex of $F_1$ is equal to the minimal vertex of $F_2$. Then there exists a unique face $F$ containing both $F_1$ and $F_2$ and satisfying $\dim F=\dim F_1+\dim F_2$. Moreover, the minimal vertex of $F$ is equal to the minimal vertex of $F_1$ and the maximal vertex of $F$ is equal to the maximal vertex of $F_2$. 
\end{lemma}
\begin{proof}
    Let $v$ be the maximal vertex of $F_1$, which is also the minimal vertex of $F_2$. By Lemma \ref{lem:dimsum=dimP}, $\dim F^-(v)+\dim F^+(v)=\dim P$. Therefore, the cone of $F^-(v)$ at $v$ and the cone of $F^+(v)$ at $v$ have complimentary dimensions in the cone of $P$ at $v$. Moreover, the latter is equal to the convex hull of the former two cones. Thus, $F^-(v)-v$ and $F^+(v)-v$ lie in skew subspaces of $P-v$. Hence, the face $F_1$ in $F^-(v)$ and the face $F_2$ in $F^+(v)$ generate a face $F$ whose cone at $v$ is equal to the convex hull of the cones of $F_1$ and $F_2$ at $v$. Therefore, $\dim F=\dim F_1+\dim F_2$. 

    Denote the minimal vertex of $F_1$ by $w_1$ and the maximal vertex of $F_2$ by $w_2$. Then both $F_1$ and $F_2$ are contained in the face $F^+(w_1)\cap F^-(w_2)$. Hence, both of the cones of $F_1$ and $F_2$ at $v$ are contained in the cone of $F^+(w_1)\cap F^-(w_2)$ at $v$. Therefore, the face $F$ must be contained in $F^+(w_1)\cap F^-(w_2)$, and the last claim of the lemma follows. 
\end{proof}
The following corollary is a special case of the lemma when $F_1$ and $F_2$ are edges.
\begin{corollary}\label{cor:2faceexist}
    Given any directed edges $v_1\to v_2$ and $v_2\to v_3$, there exists a 2-face of $P$ containing both of these edges. 
\end{corollary}
Let $(F_1, \ldots, F_k)$ be a monotone chain of faces of $P$. 
By Theorem~\ref{thm:facesofCH}, there is an associated face $[F_1, \ldots, F_k]$ of $\CH(P)$. 

By Lemma \ref{lem:bendface}, for every $1\leq i\leq k-1$, there exists a face $F_{i, i+1}$ containing $F_i$ and $F_{i+1}$ and satisfying $\dim F_{i, i+1}=\dim F_i+\dim F_{i+1}$. Moreover, the minimal vertex of $F_{i, i+1}$ is equal to the minimal vertex of $F_i$ and the maximal vertex of $F_{i, i+1}$ is equal to the maximal vertex of $F_{i+1}$. Therefore, 
$(F_1, \ldots, F_{i-1}, F_{i, i+1}, F_{i+2}, \ldots, F_k)$
is also a monotone chain of faces, and hence corresponds to a face of $\CH(P)$. 
\begin{proposition}\label{prop:numberofplusonefaces}
    Under the above notation, the faces of $\CH(P)$ that contain $[F_1,\ldots,F_k]$ as a facet are precisely the following:
\begin{itemize}
\item the faces $[F_1,\ldots,F_{i-1}, F_{i,i+1}, F_{i+2},\ldots,F_k]$ described above; and
\item the faces $[F_1,\ldots,F_{i-1}, F_i', F_{i+1},\ldots,F_k]$, where $F_i'$ is a face having $F_i$ as a facet and sharing the same minimal and maximal vertices as $F_i$.
\end{itemize}
\end{proposition}
\begin{proof}
    Denote the maximal vertex of $F_i$ by $v_i$, and denote the source of $P$ by $v_0$, which is also the minimal vertex of $F_1$. By Theorem~\ref{thm:facelatticeCH}, if a face $[G_1, \ldots, G_m]$ of $\CH(P)$ contains the face $[F_1, \ldots, F_k]$, then 
    \[
    F_1\cup\cdots \cup F_k\subset G_1\cup \cdots \cup G_m.
    \]
    In this case, the maximal and minimal vertices of each $G_i$ must be a vertex in $\{v_0, \ldots, v_k\}$. In particular, $m\leq k$. 

  Since each face of $P$ also satisfies Assumption~\ref{ass:equivalent}, we can apply Lemma~\ref{lem:dimatmostP} to each $G_i$ and conclude that 
    \[
    \sum_{1\leq i\leq k}\dim F_i\leq \sum_{1\leq i\leq m}\dim G_i. 
    \]
    By the dimension formula of the faces of $\CH(P)$ in Theorem \ref{thm:facesofCH}, if the face $[G_1, \ldots, G_m]$ is one dimension larger than $[F_1, \ldots, F_k]$, then either
    \[
    \sum_{1\leq i\leq k}\dim F_i= \sum_{1\leq i\leq m}\dim G_i \quad\text{and}\quad m=k-1
    \]
    or
    \[
    \sum_{1\leq i\leq k}\dim F_i= \left(\sum_{1\leq i\leq m}\dim G_i\right)-1\quad\text{and}\quad m=k.
    \]
    In the first case, the sequence $(G_1, \ldots, G_m)$ must be equal to $(F_1, \ldots, F_{i-1}, F_{i, i+1}, F_{i+2}, \ldots, F_k)$ for some $i$. In the second case, the sequence $(G_1, \ldots, G_m)$ must be of the form $(F_1, \ldots, F_{i-1}, F_i', F_{i+1}, \ldots, F_k)$ as described in the proposition. 

    Conversely, it follows immediately from Theorem~\ref{thm:facelatticeCH} that every face $[F_1, \ldots, F_{i-1}, F_{i, i+1}, F_{i+2}, \ldots, F_k]$ and $[F_1, \ldots, F_{i-1}, F_i', F_{i+1}, \ldots, F_k]$ contains $[F_1, \ldots, F_k]$ as a facet. 
\end{proof}


\begin{theorem}\label{thm:simpleCH}
    Assume that $P$ satisfies Assumption \ref{ass:equivalent}. 
    Then the following conditions are equivalent:
    \begin{enumerate}
    \item $\CH(P)$ is simple.
    \item For every directed edge $v\to w$ of $P$, the number of triangles of the form $(v, w, x)$, where $v\to x$ and $x\to w$ are directed edges, is equal to $\dim F^-(w)\cap F^+(v)-1$.
    \item For every directed edge $v\to w$ of $P$, the face $F^-(w)\cap F^+(v)$ is simple along the edge $v\to w$, that is, the star of the edge $v\to w$ in $F^-(w)\cap F^+(v)$ is simple. 
    \item For every positive dimensional face $F$ of $P$ with minimal vertex $v$ and maximal vertex $w$, the star of $F$ within the face $F^+(v)\cap F^-(w)$ is simple. 
    \end{enumerate}
\end{theorem}
\begin{proof}
Clearly, $(4)\Rightarrow (3)$. 

Now, we prove $(2)\Leftrightarrow(3)$. The face $F^-(w)\cap F^+(v)$ is simple along the edge $v\to w$ if and only if the number of 2-dimensional faces of $F^-(w)\cap F^+(v)$ containing $v\to w$ is equal to $\dim(F^-(w)\cap F^+(v))-1$. 
By Proposition~\ref{prop:constructexamples}~(3), every 2-face of
$F^-(w)\cap F^+(v)$ containing $v\to w$ must be a triangle of the form $(v, w, x)$ with $v\to x$ and $x\to w$ being directed edges (see also Example \ref{exm:triandquad}). Therefore, $(2)\Leftrightarrow(3)$. 

Next, we prove $(2)\Rightarrow (1)$. 
By Corollary \ref{cor:vertofCHP}, a vertex of $\CH(P)$ is of the form $[v_0\to  \cdots\to  v_k]$, where $v_0$ is the source and $v_k$ is the sink. 
Fixing such a vertex of $\CH(P)$, we compute the number of edges in $\CH(P)$ connecting to this vertex using Corollary \ref{cor:edgesadjacenttovertexofCHP}.
By Corollary \ref{cor:2faceexist}, there are $k-1$ edges of the second kind in Corollary \ref{cor:edgesadjacenttovertexofCHP}.

Assuming condition (2) holds, 
the number of edges of the first kind in Corollary \ref{cor:edgesofCHP} is equal to 
\[
\sum_{1\leq i\leq k}\left(\dim F^-(v_i)\cap F^+(v_{i-1})-1\right)=\left(\sum_{1\leq i\leq k}\dim F^-(v_i)\cap F^+(v_{i-1})\right)-k.
\]
By Lemma \ref{lem:dimensionsbehavenicely}, $\dim F^+(v_{i-1})\cap F^-(v_i)=\dim F^-(v_i)- \dim F^-(v_{i-1})$. 
Therefore, 
\begin{align*}
\left(\sum_{1\leq i\leq k}\dim F^-(v_i)\cap F^+(v_{i-1})\right)-k&=\left(\sum_{1\leq i\leq k}\left(\dim F^-(v_i)- \dim F^-(v_{i-1})\right)\right)-k\\
&=\dim F^-(v_k) - \dim F^-(v_0)-k\\
&=\dim P-k.
\end{align*}

Therefore, the total number of edges in $\CH(P)$ containing $[v_0\to \cdots \to v_k]$ is equal to 
\[
(k-1)+\dim P-k=\dim P-1=\dim \CH(P).
\]
Therefore, $[v_0\to \cdots \to v_k]$ is a simple vertex of $\CH(P)$. Since $[v_0\to \cdots \to v_k]$ is an arbitrary vertex of $\CH(P)$, $\CH(P)$ is simple. 

Finally, we prove $(1)\Rightarrow (4)$. Let $F$ be a face of $P$ whose minimal and maximal vertices are $v$ and $w$, respectively. Since $\CH(P)$ is simple, the number of faces of $\CH(P)$ containing $[F^-(v), F, F^+(w)]$ as a facet is equal to 
$\dim \CH(P)-\dim [F^-(v), F, F^+(w)]$. By Theorem~\ref{thm:facesofCH} and Lemma \ref{lem:dimensionsbehavenicely}, we have
\begin{align*}
&\dim \CH(P)-\dim [F^-(v), F, F^+(w)]\\
=&\dim P-1-(\dim F^-(v)+\dim F+\dim F^+(w)-3)\\
=&(\dim F^-(v)+\dim F^+(v)\cap F^-(w)+\dim F^+(w))-1-(\dim F^-(v)+\dim F+\dim F^+(w)-3)\\
=&\dim F^+(v)\cap F^-(w)-\dim F+2.
\end{align*}

On the other hand, we can also compute the number of faces of $\CH(P)$ containing $[F^-(v), F, F^+(w)]$ as a facet using Proposition \ref{prop:numberofplusonefaces}. There are exactly 2 such faces of the first kind in Proposition \ref{prop:numberofplusonefaces}, and the number of faces of the second kind is equal to the number of faces in $F^+(v)\cap F^-(w)$ containing $F$ as a facet. Combining the above arguments, we conclude that the number of faces in $F^+(v)\cap F^-(w)$ containing $F$ as a facet is equal to 
\[
\dim F^+(v)\cap F^-(w)-\dim F.
\]
This is exactly equivalent to the star of $F$ within $F^+(v)\cap F^-(w)$ being simple. 
\end{proof}

Since simple polytopes satisfy Assumption~\ref{ass:irred} and automatically fulfill condition~(3) of Theorem~\ref{thm:simpleCH}, the following corollary is immediate.
\begin{corollary}\label{cor:Psim then CH sim}
    If $P$ is a simple polytope and $O^-$ is a stratification, then $\CH(P)$ is simple.
\end{corollary}

In the rest of the section we will present some operations on polytopes preserving condition that the monotone path polytope is simple.

\begin{proposition}\label{prop:construct simple examples}
Suppose all the following polytopes satisfy Assumption~\ref{ass:equivalent}.
    \begin{enumerate}
\item   Let $(P_1,\ell_1)$ and $(P_2,\ell_2)$  be such that $\CH(P_1)$ and $\CH(P_2)$ are simple, Then $(P_1\times P_2,\ell_1\times\ell_2)$ satisfies Assumption~\ref{ass:equivalent} and $\CH(P_1\times P_2,\ell_1\times\ell_2)$ is simple.
\item Let $P\subset \RR^n$ and $\ell: \RR^n\to \RR$ be such that every vertex $v$ is simple in $F^-(v)$ (resp. in $F^+(v)$). Then the pyramid $Pyr(P)\subset\RR^{n+1}$ over $P$ with the apex that is the new source (resp. sink) satisfies Assumption~\ref{ass:equivalent} and has simple Chow quotient. 
\item If $\CH(P)$ is simple then for  any face $F$, $\CH(F)$ is simple.
 \end{enumerate}
\end{proposition}
\begin{proof}
    Parts $(1)$ and $(3)$ are straightforward, we will prove part $(2)$. We assume that every vertex $v$ is simple in $F^+(v)$ and that the apex $a$ of the pyramid $Pyr(P)$ is the new sink, the proof in the second case is the same. 

    First, by part $(2)$ of Proposition~\ref{prop:constructexamples} $Pyr(P)$ satisfies Assumption~\ref{ass:equivalent}. Now, by Theorem~\ref{thm:simpleCH}, for simplicity of monotone path polytope, it is enough to show that for every directed edge $v\to w$ of $Pyr(P)$ we have that the number of triangles $(v,w,x)$ where $v\to x$ and $x\to w$ are directed edges, is equal to $\dim F^-(w)\cap F^+(v)-1$. Since $\CH(P)$ is simple, it is enough only to consider directed edges $v\to a$. Moreover, after replacing $P$ with $F^+(v)$ we can assume that $v$ is the minimal vertex of $P$. In that case the number of the triangles $(v,w,a)$ in $Pyr(P)$ as above is equal to the degree of $v$ in $P$ since any vertex of $P$ is connected to the apex of the $Pyr(P)$. On the other hand, we have $F^-(a)\cap F^+(v) = Pyr(P)$, so the statement follows by simplicity of vertex $v$ in~$P$.
\end{proof}

\begin{corollary}
    Let $P$ be a polytope obtained by taking alternatively taking products and iterated upper (or lower) pyramid over simple polytopes satisfying Assumption~\ref{ass:equivalent}. Then $\CH(P)$ is simple.
\end{corollary}
\begin{proof}
 An iterated upper (or lower) pyramid over simple polytope satisfying Assumption~\ref{ass:equivalent} satisfies assumptions of part $(2)$ of Proposition~\ref{prop:construct simple examples}. Moreover, if two polytopes $P,Q$ satisfy assumptions of part $(2)$ of Proposition~\ref{prop:construct simple examples}, then their product also satisfies assumptions of part $(2)$ of Proposition~\ref{prop:construct simple examples}. Hence the statement follows by induction on the number of operations.
\end{proof}
\begin{remark}
    Notice that unlike part $(2)$ of Proposition~\ref{prop:constructexamples}, in part $(2)$ of Proposition~\ref{prop:construct simple examples} we could only take a pyramid provided that either maximal or minimal vertex of $P$ is simple. The following example shows that we could not expect more general statement of part $(2)$ of Proposition~\ref{prop:construct simple examples}.
\end{remark}

\begin{example}\label{ex:doublepyramid}
    Let us consider the polytope $P$ this is a $3$-dimensional pyramid over a square. We assume that the function $l$ attains minimum at the apex $v$ of the pyramid. In this setting Assumption \ref{ass:equivalent} holds for $P$ and clearly $\CH(P)$ is simple. Let now $P'$ be a $4$-dimensional pyramid over $P$ assuming that $\ell$ attains maximum at the apex $w$ of $P'$. By Proposition~\ref{prop:constructexamples}, Assumption \ref{ass:equivalent} holds also for $P'$. However, $\CH(P')$ is not simple at the vertex corresponding to the edge joining $v$ and $w$. Indeed, there are four triangles containing that edge, each one containing one of the four edges of $P$ that are adjacent to $v$. 
    
    Notice that in order to construct $P'$, we took pyramid upper pyramid over polytope $P$ with non-simple lower vertex. Thus part $(2)$ of Proposition~\ref{prop:construct simple examples} is not applicable in that case.
\end{example}

\section{$\cP$-kernels, Kazhdan-Lusztig-Stanley polynomials and Chow polynomials}\label{sec:ker and KLS poly}

In his seminal work~\cite{stanley1992subdivisions}, Stanley generalized the theory of Kazhdan–Lusztig polynomials to arbitrary posets and introduced the notion of a kernel of a poset. In~\cite{Proudfoot}, Proudfoot provided a geometric interpretation of these generalized Kazhdan–Lusztig polynomials via point counting over finite fields. In this section, we present a new geometric perspective on poset kernels and Kazhdan–Lusztig polynomials, arising from Białynicki–Birula theory. Although a complete characterization of when Białynicki–Birula theory produces a poset equipped with a kernel and Kazhdan–Lusztig polynomials is currently unknown, this framework is well understood for projective toric varieties, and more generally, for convex polytopes.

\subsection{Kazhdan-Lusztig-Stanley theory} Let us first briefly recall basics of Kazhdan-Lusztig-Stanley theory. We follow closely the exposition in \cite[Sections~2 and~3]{ferroni2024chow}  and refer reader there as well as to \cite[Section~2]{Proudfoot} for details.  Throughout this section all posets are assumed to be finite.
Let  $(X,\leq)$ be a finite poset, for $s,t\in X$ with $s\leq t$, we will denote by $[s,t]=\{v\in X \,|\, s\leq v\leq t\}$ the interval between $s$ and $t$.
The incidence algebra $\cI(X)$ of $(X,\leq)$ is a free $\Z[x]$-module over the set of all intervals of $X$. In other words, an element $a\in \cI(X)$ associates a polynomial $a_{st}(x)\in \Z[x]$ to every closed interval $[s,t]$ in $X$. The product of $\cI$ is defined via
\[
(ab)_{st}(x) = \sum_{s\leq w\leq t} a_{sw}(x)\, b_{wt}(x) \qquad \text{ for every $s\leq t$ in $X$}.
\]


An important interpretation of the incidence algebra is that it can be realized as a subalgebra of the algebra of $|X|\times |X|$ matrices with polynomial entries. 
Indeed, to each element $a\in \cI(X)$ we may associate a matrix $A=(a_{st})_{s,t\in X}$, in which case the product in the incidence algebra corresponds to matrix multiplication. 
In particular, this shows that the incidence algebra is associative but, in general, non-commutative, with multiplicative identity given by
\[
\delta_{st} =
\begin{cases}
1 & \text{if } s = t,\\
0 & \text{if } s < t.
\end{cases}
\]
Choosing a linear extension of the partial order $\le$ to index the rows and columns, the matrix associated to any element $a\in \cI(X)$ is upper triangular. 
This yields the following well-known characterization of the invertible elements of $\cI(X)$.

\begin{lemma}\label{lem:invertible}
    An element $a\in \cI(X)$ admits a two-sided inverse $a^{-1}\in \cI(X)$ if and only if $a_{ss}=\pm1$ for every $s\in X$.
\end{lemma}

Consider a grading $\rho\colon X\to \Z$ compatible with the partial order, that is, $\rho(s)\leq \rho(t)$ whenever $s\leq t$. Denote $\rho_{st}= \rho(t)-\rho(s)$ for any $s\leq t$ in $X$. We define a subalgebra $\cI_\rho(X)$ of the incidence algebra $\cI(X)$ by
\[
\cI_{\rho}(X) = \left\{a\in \cI(X) \,|\, \deg a_{st}(x) \leq \rho_{st} \text{ for all $s\leq t$ in $X$}\right\}.
\]

We define an involution $a\mapsto a^{\rev}$ on $\cI_{\rho}(X)$ via\footnote{In \cite{Proudfoot}, $\f^{\rev}$ is denoted by $\bar{\f}$ instead. }
\[
  \left(a^{\rev}\right)_{st}(x) = x^{\rho_{st}}\, a_{st}(x^{-1}).
\]
One can show that the above involution is compatible with multiplication, in the sense that $(ab)^{\rev}=a^{\rev}b^{\rev}$, and it commutes with inversion: $(a^\rev)^{-1}=(a^{-1})^\rev$ whenever $a\in \cI(X)$ is invertible.

A central notion in Kazhdan–Lusztig–Stanley theory is the concept of an $(X,\rho)$-kernel.

\begin{definition}
    An element $\kappa\in \cI_\rho(X)$ is called an $(X,\rho)$-kernel if $\kappa_{ss}=1$ for all $s\in X$ and
    \[
    \kappa^{-1}=\kappa^\rev.
    \]
\end{definition}

For any graded poset $X$ one can define its \emph{characteristic kernel}. First recall, the definition of the \emph{M\"obius function} of $X$:
    \[ \mu_{st} = \begin{cases}
        1 & \text{ if $s=t$},\\
        - \sum_{s\leq w < t} \mu_{sw} & \text{ if $s < t$}.
    \end{cases}\]
 Then the characteristic kernel $\chi$ is defined by 
    \[ 
    \chi_{st}(x) = \sum_{s\leq w\leq t} \mu_{sw}\,x^{\rho_{wt}}.
    \]
For any closed interval $[s,t]$ in $X$. Alternatively, we can define $\chi$ via
\[
\chi = \mu\cdot \zeta^{\rev} = \zeta^{-1}\cdot \zeta^{\rev},
\]
where $\zeta_{st} = 1$ for all $s\leq t$.
More generally, for every invertible element $a\in \cI_\rho(X)$, the element $\kappa =a^{-1}a^\rev$ is a kernel. Indeed, we have $\kappa_{ss}=1$ for any $s\in X$, and
\[
\kappa^{-1}= (a^{-1}a^\rev)^{-1}= (a^\rev)^{-1} a = (a^{-1})^\rev (a^\rev)^\rev = (a^{-1}a^\rev)^\rev = \kappa^\rev.
\]
 Stanley proved in \cite[Theorem~6.5]{stanley1992subdivisions} that all $(P,\rho)$-kernels arise in this way.

\begin{theorem}\label{thm:kls-functions}
    Let $\kappa\in \cI_{\rho}(X)$ be a $(X,\rho)$-kernel. There exists a unique element $f\in \cI(X)$ satisfying the following properties:
    \begin{enumerate}
        \item $f_{ss}(x) = 1$ for all $s\in X$.
        \item  $\deg f_{st}(x) < \frac{1}{2} \rho_{st}$ for all $s < t$.
        \item  $f^{\rev} = \kappa\cdot f$.
    \end{enumerate}
    Similarly, there exists a unique element $g\in \cI(X)$ satisfying the following properties:
    \begin{enumerate}
        \item[(1')] $g_{ss}(x) = 1$ for all $s\in X$.
        \item[(2')]  $\deg g_{st}(x) < \frac{1}{2} \rho_{st}$ for all $s < t$.
        \item[(3')]  $g^{\rev} = g \cdot \kappa$.
    \end{enumerate}
\end{theorem}
Elements $f, g\in \cI(X)$  are respectively called right and left Kazhdan-Lusztig-Stanley functions associated to $\kappa$. The following lemma is a standard fact about kernels.
\begin{lemma}[{\cite[Lemma 3.1]{ferroni2024chow}}]
 Let $\kappa$ be a $(X,\rho)$-kernel. Then, for every $s < t$ in $X$, the polynomial $\kappa_{st}(x)$ is divisible by $x - 1$. 
\end{lemma}
The next two definitions are introduced in \cite[Section 3]{ferroni2024chow}. 
\begin{definition}
    The reduced kernel $\overline{\kappa}
\in \cI_\rho(X)$ associated to $(X,\rho)$-kernel $\kappa$ is defined via
    \[ \overline{\kappa}_{st}(x) 
    = \begin{cases}
        \frac{1}{x-1}\, \kappa_{st}(x) & \text{ if $s < t$}\\
        -1 & \text{ if $s = t$}.
    \end{cases}\]
\end{definition}


\begin{definition}\label{def:chow-function}
    Let $\kappa$ be a $(X,\rho)$-kernel. We define the Chow polynomial associated to $\kappa$, or the $\kappa$-Chow polynomial, as the element $\HH\in \mathcal{I}_{\rho}(X)$ defined by
    \[ 
    \HH = - \left(\overline{\kappa}\right)^{-1}.
    \]
\end{definition}
Note that, by Lemma~\ref{lem:invertible}, the reduced kernel is an invertible element of the incidence algebra, and therefore Definition~\ref{def:chow-function} is well defined.
Furthermore, we note that Definition~\ref{def:chow-function} is equivalent to either of the following properties (see \cite[Subsection 3.2]{ferroni2024chow}).  For all $s<t$ in $P$,
    \begin{align}
        \HH_{st}(x) &= \sum_{s < w \leq t} \overline{\kappa}_{sw}(x)\, \HH_{wt}(x)\qquad \\
        \HH_{st}(x) &= \sum_{s\leq w < t} \HH_{sw}(x)\, \overline{\kappa}_{wt}(x).
    \end{align}

Moreover, authors of \cite{ferroni2024chow} show that the Chow polynomials have nice properties.
\begin{theorem}[{\cite[Theorem 1.1]{ferroni2024chow}}]\label{thm:non-negative unimodal}
        Let $\kappa$ be a $(P,\rho)$-kernel. If the right KLS function $f$ or the left KLS function $g$ is non-negative, then the Chow polynomial $\HH$ is non-negative and unimodal.
\end{theorem}

\begin{theorem}[{\cite[Theorem 1.4]{ferroni2024chow}}]\label{thm:gammapos}
           Let $P$ be any Cohen--Macaulay poset. The $\chi$-Chow polynomial of $P$ is $\gamma$-positive.
\end{theorem}

\subsection{Kernel for the poset $\cO$}
Let $P$ be a polytope and let $\ell$ be a linear function on $P$ such that $\ell$ is nonconstant on every positive dimensional face. 
Throughout this section, we work under Assumption \ref{ass:equivalent}, and we do not assume that $P$ is simple.
We will be working with two posets associated to $P$: the poset of vertices $\cO$ as discussed before and the face poset of $P$. 
We use $\prec$ and $\preceq$ to denote the order relations in the vertex poset, and $<$ and $\le$ to denote the order relations in the face poset.

By Remark \ref{rem:O is subposet of face poset}, the vertex poset can be regarded as a subposet of the face poset. 
Given two vertices $v\preceq w$ of $P$, we denote the face $F^-(w)\cap F^+(v)$ by $F_{[v, w]}$. Given a face $F$ of $P$, we use $\max F$ and $\min F$ to denote the maximal and minimal vertices in $F$, respectively. 
For any two vertices $v,w$ of $P$ let us define the set 
 \[
\mathcal{S}_{vw}\coloneqq \{F\mid v=\min F, w=\max F\}.
\]
Since the intersection of faces is a face, we have the following lemma. 
\begin{lemma}
    Let $v\prec w$ be two vertices of P. The set of faces 
$\mathcal{S}_{vw}$
    contains a unique minimal face, which we denote by $F^{\min}_{[v,w]}$. Moreover, $\mathcal{S}_{vw}$ consists of all faces that is between $F_{[v, w]}^{\min}$ and $F_{[v, w]}$ in the face poset. 
\end{lemma}

Next, we recall basic facts about Kazhdan-Lusztig-Stanley theory the face poset. For more details, see \cite[Example 2.14]{Proudfoot}) .
The face poset of a polytope always has a grading given by dimension of a face and a kernel defined by
\[
\lambda_{FG}=(x-1)^{\dim G-\dim F}.
\]
Given a polytope $P$, and any two faces $F\leq G$, there is an $f$-polynomial $\f_{FG}\in \Z[x]$, with degree at most $(\dim G-\dim F-1)/2$. These are the right Kazhdan-Lusztig-Stanley polynomials associated to the kernel $\lambda$. We use a different font to emphasize it is the $f$-polynomial of the polytope face poset, and we use the convention that $\dim \emptyset =-1$. 
Since the $f$-function satisfies Condition~(3) of Theorem~\ref{thm:kls-functions}, we have
\begin{equation}\label{eq_lambda}
    \f_{F_1F_2}^{\rev}=\sum_{F_1\leq F\leq F_2} \lambda_{F_1 F}\f_{F F_2}=\sum_{F_1\leq F\leq F_2}(x-1)^{\dim F-\dim F_1}\f_{F F_2}.
\end{equation}

Passing to the vertex subposet we define two families of functions: $g_{vw}\coloneqq \f_{w F_{[v,w]}}$ and $f_{vw}\coloneqq\f_{vF_{[v,w]}}$.
\begin{proposition}\label{prop:kernel}
    If $\CH(P)$ is simple, then 
    \begin{enumerate}
        \item the polynomials
        \[
        \kappa_{vw}\coloneqq \sum_{F\in \mathcal{S}_{vw}}(x-1)^{\dim F}
        \]
        define a kernel on the vertex poset;
        \item the polynomials $g_{vw}\coloneqq \f_{w F_{[v,w]}}$ and $f_{vw}\coloneqq\f_{vF_{[v,w]}}$ are the left and right Kazhdan-Lusztig-Stanley polynomials of the interval $[v, w]$ associated to the kernel $\kappa$.
    \end{enumerate}
\end{proposition}
\begin{proof}
    Clearly, $f_{vv}=g_{vv}=1$. Moreover, for any $v\prec w$, both $\deg f_{vw}, \deg g_{vw}<(\rk(w)-\rk(v))/2$. Thus, by \cite[Theorem 2.5]{Proudfoot}, it suffices to show that
    \begin{equation*}
        f_{vw}^{\rev}=\sum_{v\preceq u\preceq w} \kappa_{vu} f_{uw}
    \end{equation*}
    and
    \[
    g_{vw}^{\rev}=\sum_{v\preceq u\preceq w} g_{vu}\kappa_{uw}.
    \]
The two equations only differ by symmetry, so we only prove the first one. Replacing $P$ by the face $F_{[v, w]}$, we can also assume that $v=\hat{0}$ and $w=\hat{1}$. 
We start with the right-hand side:
\begin{equation}\label{eq_long}
    \sum_{\hat{0}\preceq u\preceq \hat{1}} \kappa_{\hat{0}u} f_{u\hat{1}}=\sum_{\hat{0}\preceq u\preceq \hat{1}} \left(\sum_{F\in \mathcal{S}_{\hat{0}u}}(x-1)^{\dim F}\right)\f_{u F_{[u, \hat{1}]}}=\sum_{\hat{0}\preceq u\preceq \hat{1}} \sum_{F\in \mathcal{S}_{\hat{0}u}}(x-1)^{\dim F}\f_{u F_{[u, \hat{1}]}}.
\end{equation}
\begin{lemma}
    For any $F\in \mathcal{S}_{\hat{0}u}$, $\f_{F P}=\f_{u F_{[u, \hat{1}]}}$. 
\end{lemma}
\begin{proof}[Proof of Lemma]
    By the proof of Lemma \ref{lem:facetsofCHPtwotypes}, we know that the cone of $P$ at $u$ is isomorphic to the product of the cone of $F_{[\hat{0}, u]}$ at $u$ and the cone of $F_{[u, \hat{1}]}$ at $u$. Thus, 
    \[
    \f_{u F_{[u, \hat{1}]}}=\f_{F_{[\hat{0}, u]}P}.
    \]
    Notice that the $f$-polynomial of a product of poset is equal to the product of the $f$-polynomial of the posets. Thus, by the above product structure, we also have that for $F\in \mathcal{S}_{\hat{0}u}$,
    \begin{equation}\label{eq_true}
            \f_{F P}=\f_{F F_{[\hat{0}, u]}} \f_{u F_{[u, \hat{1}]}}.
    \end{equation}
    By Theorem \ref{thm:simpleCH}, the star of $F$ within $F_{[\hat{0}, u]}$ is simple, and hence $\f_{F F_{[\hat{0}, u]}}=1$. Therefore, we obtain the desired equation of the lemma. 
\end{proof}
Now we can simplify the formula \eqref{eq_long} for the right-hand side of the desired equality:
\begin{align*}
    \sum_{\hat{0}\preceq u\preceq \hat{1}}\sum_{F\in \mathcal{S}_{\hat{0}u}}(x-1)^{\dim F} \f_{u F_{[u, \hat{1}]}} &=\sum_{\hat{0}\preceq u\preceq \hat{1}} \sum_{F\in \mathcal{S}_{\hat{0}u}}(x-1)^{\dim F} \f_{F P}\\
    &=\sum_{\hat{0}\leq F\leq P} (x-1)^{\dim F} \f_{F P}\\
    &=\f_{\hat{0}, P}^{\rev}\\
    &=f_{\hat{0}\hat{1}}^{\rev}
\end{align*}
where the first equality follows from the preceding lemma, the third equality follows from \eqref{eq_lambda}, and the last one follows from the definition of the $f$-polynomial $f_{\hat{0}\hat{1}}$. 
\end{proof}

\begin{remark}
    It is not hard to show that the kernel $\kappa$ could be alternatively defined using the kernel of the face pocet $\lambda$ by
    \[\kappa_{vw}=\sum_{\substack{F^-(v)\subset \tilde F\\ \max \tilde F=w}}\lambda_{F^-(v),\tilde F}.\]
\end{remark}

\subsection{Chow polynomials for the poset $\mathcal{O}$}
The main theorem of this section relates the Chow polynomial of the poset $\mathcal{O}$ with the $h$-polynomial of the dual of the monotone path polytope. For the definition of $h$-polynomial and for a comparison result, see \cite[Subsection 2.5 and Theorem 5.3]{ferroni2024chow}.

\begin{theorem}\label{thm:Chow poly is h of Chow}
    If Assumption \ref{ass:equivalent} holds, then the Chow polynomial $H_{vw}$ for the poset $\mathcal{O}$ with kernel $\kappa$ is the $h$-polynomial of the dual polytope $\CH\left(F^+(v)\cap F^-(w)\right)^*$. 
\end{theorem}
\begin{proof}
    For any $u\leq v$ in the poset $\OO_P$, we denote by $H'_{uv}$ the $h$-polynomial of $\CH\left(F^+(u)\cap F^-(v)\right)^*$. By the definition of Chow polynomial, we need to show that in the incidence algebra
    \[
    H'=-(\overline{\kappa})^{-1}.
    \]
    Equivalently, we need to show that, for any $u\leq v$,
    \[
    \sum_{u\leq w\leq v} \overline{\kappa}_{uw}H'_{wv}=-\delta_{uv}
    \]
    where $\delta_{uv}=1$ if $u=v$; and $\delta_{uv}=0$ otherwise. The equation is obvious when $u=v$. Thus, replacing $P$ by $F^+(u)\cap F^-(v)$, it suffices to show
    \begin{equation}\label{eq_enough}
    \sum_{\hat{0}\leq w\leq \hat{1}} \overline{\kappa}_{\hat{0}w}H'_{w\hat{1}}=0
    \end{equation}
    for a positive-dimensional polytope satisfying Assumption \ref{ass:equivalent}. 
    
Let $Q$ be a polytope of dimension $d+1$ satisfying Assumption~\ref{ass:equivalent}.
By Theorem~\ref{thm:facesofCH}, the number of $i$-dimensional faces of $\CH(Q)$
is equal to the number of sequences of faces $(F_1,\dots,F_l)$ satisfying:
\begin{enumerate}
  \item the minimal vertex of $F_1$ is $\hat{0}$;
  \item the maximal vertex of $F_l$ is $\hat{1}$;
  \item the maximal vertex of $F_k$ equals the minimal vertex of $F_{k+1}$
        for all $1\le k\le l-1$;
  \item $\dim F_1+\cdots+\dim F_l-l=i$.
\end{enumerate}

    By definition, the $h$-polynomial of $\CH(Q)^*$ can be expressed as
    \[
    h_{\CH(Q)^*}
    =\sum_{i=0}^d f_{i-1}(\CH(Q)^*)(x-1)^{d-i}
    =\sum_{i=0}^d f_{d-i-1}(\CH(Q)^*)(x-1)^{i}=\sum_{i=0}^d f_{i}(\CH(Q))(x-1)^{i}, 
    \]
    where $f_{i}$ denotes the number of $i$-dimensional faces of a polytope.\footnote{The convention is that for any $d$-dimensional polytope, it has one $(-1)$-dimensional face (the empty face) and one $d$-dimensional face.} By the above arguments, $f_{i}(\CH(Q))$ is equal to the number of sequences of faces $(F_1, \dots, F_l)$ satisfying the above properties (1-4). Therefore, 
    \begin{equation}\label{eq_sum1}
    h_{\CH(Q)^*}=\sum_{(F_1, \dots, F_l)}(x-1)^{(\sum_{1\leq j\leq l}\dim F_j)-l}
    \end{equation}
    where the sum is over all sequences of faces $(F_1, \dots, F_l)$ satisfying the above properties (1-3). Denote the maximal vertex of $F_1$ by $w$, we can rewrite the above sum as the following
    \begin{equation}\label{eq_sum2}
    h_{\CH(Q)^*}=\sum_{w \neq \hat{0}}\left(\sum_{F_1}(x-1)^{\dim F_1-1}\right)\left(\sum_{(F_2, \dots, F_l)}(x-1)^{(\sum_{2\leq j\leq l}\dim F_j)-l+1}\right)        
    \end{equation}
    where  the first sum is over all vertices $w$ of $P$ other than $\hat{0}$, the second sum is over all faces $F_1$ whose minimal vertex is $\hat{0}$ and maximal vertex is $w$, the last sum is over all sequences of faces $(F_2, \dots, F_l)$ satisfying
    \begin{enumerate}[label=(\arabic*')]
  \item the minimal vertex of $F_2$ is $w$;
  \item the maximal vertex of $F_l$ is $\hat{1}$;
  \item the maximal vertex of $F_k$ equals the minimal vertex of $F_{k+1}$
        for all $2\le k\le l-1$.
    \end{enumerate}
    When $l=1$, we set the value of the last sum to be 1. 
    It follows from the definition of $\overline{\kappa}$ we get
    \[
    \sum_{F_1}(x-1)^{\dim F_1-1}=\overline{\kappa}_{\hat{0}w}.
    \]
    Thus, applying \eqref{eq_sum1} to the last sum of \eqref{eq_sum2}, we have
    \[
    h_{\CH(Q)^*}=\sum_{w \neq \hat{0}}\overline{\kappa}_{\hat{0}w}\cdot h_{\CH(F^+(w))^*}.
    \]
    Equivalently, we have
    \[
    H'_{\hat{0}\hat{1}}=\sum_{\hat{0}< w\leq \hat{1}} \overline{\kappa}_{\hat{0}w}H'_{w\hat{1}}.
    \]
    Since $\overline{\kappa}_{\hat{0}\hat{0}}=-1$, the above equation is equivalent to
    \[
    \sum_{\hat{0}\leq w\leq \hat{1}} \overline{\kappa}_{\hat{0}w}H'_{w\hat{1}}=0,
    \]
    which is \eqref{eq_enough}. Hence, we have finished the proof. 
\end{proof}

As an immediate corollary we obtain the following result.
\begin{corollary}
     If Assumption \ref{ass:equivalent} holds for $P$ and $\CH(P)$ is simple, the Chow polynomial $H_{vw}$ for the poset $\mathcal{O}$ with kernel $\kappa$ has positive coefficients and is palindromic and unimodal.
\end{corollary}
\begin{remark}
Assume that $P$ satisfies Assumption~\ref{ass:equivalent}. 
When $\CH(P)$ is not simple, the map $\kappa$ from Proposition~\ref{prop:kernel} need not be a kernel. Nevertheless, we may still define the Chow polynomials $H$ via $-(\overline{\kappa})^{-1}$, which may not be palindromic. In this case, the dual polytopes $\CH\!\left(F^+(v)\cap F^-(w)\right)^*$ need not be simple.
We may define the $h$-polynomial of a non-simple polytope $Q$ in terms of its $f$-vector by the standard relation
\[
h_Q(x)=\sum_{i=0}^d f_{i-1}(Q)(x-1)^{d-i}.
\]
In this setting, Theorem~\ref{thm:Chow poly is h of Chow} remains valid with the same proof.
\end{remark}

\subsection{Polytopes for which kernel $\kappa$ is the characteristic kernel}


As illustrated by Theorem~\ref{thm:gammapos}, Chow polynomials exhibit better behavior for the characteristic kernel. 
In this section, we provide conditions that guarantee that the kernel $\kappa$ defined in Proposition~\ref{prop:kernel} coincides with the characteristic kernel of the poset $\mathcal{O}$.

\begin{proposition}\label{prop:example characteristic kernel}
Assume that all polytopes here satisfy Assumption~\ref{ass:equivalent} and have simple monotone path polytopes.
\begin{enumerate}
    \item Let $P$ and $Q$ be polytopes such that the kernels $\kappa_P$ and $\kappa_Q$ recover the characteristic kernels of $\mathcal{O}_P$ and $\mathcal{O}_Q$, respectively. Then the kernel $\kappa$ for the product $P \times Q$ recovers the characteristic kernel of $\mathcal{O}_{P \times Q}$.
    
    \item Let $(P,\ell)$ be such that every vertex $v$ is simple in $F^-(v)$, and let $\mathrm{Pyr}(P)$ be a lower pyramid over $P$. Then the kernel $\kappa_{\mathrm{Pyr}(P)}$ recovers the characteristic kernel of $\mathcal{O}_{\mathrm{Pyr}(P)}$.
    
    \item The kernel $\kappa$ of a simplex is the characteristic kernel.
\end{enumerate}
\end{proposition}
\begin{proof}
For part $(1)$, it suffices to observe that both $\kappa$ and the characteristic kernels are multiplicative with respect to products of polytopes.

For part $(3)$, we argue by direct computation. For any vertices $v \leq w$, the intersection $F^+(v)\cap F^-(w)$ is a simplex of lower dimension. Thus, it suffices to verify that
    \[ \kappa_{\hat{0}\hat{1}}(x) = \chi_{\hat{0}\hat{1}}(x).
    \]
The poset $\mathcal{O}$ of the simplex is a chain of length $n+1$ so we get  $\chi_{\hat{0}\hat{1}}(x) = x^n-x^{n-1}$.
    On the other hand we get that the number of faces of dimension $i$ in $\mathcal{S}_{\hat{0}\hat{1}}$ is given by $\binom{n-1}{i-1}$ and thus
    \[
\kappa_{\hat{0}\hat{1}}(x)  = \sum_{i=1}^n \binom{n-1}{i-1}(x-1)^i = (x-1)\sum_{j=0}^{n-1} \binom{n-1}{j}(x-1)^j=(x-1)(1+(x-1))^{n-1}=x^n-x^{n-1}. 
    \]

It remains to prove part $(2)$. By assumption, every vertex $v$ of $P$ is simple in $F^-(v)$, and the apex $a$ of $Pyr(P)$ is the new minimum of the linear function. It suffices to check that the kernels coincide for intervals of the form $[a,v]$. By replacing $P$ with $F^-(v)$, we may assume that $v$ is the maximal vertex of $P$. Let $w$ denote the minimal vertex of $P$. Then the Möbius function satisfies $\mu_{aa}=1$, $\mu_{aw}=-1$, and $\mu_{ax}=0$ for any other vertex $x$. It follows that
\[
\chi_{a,v} = x^{\dim P+1}-x^{\dim P}.
\]
On the other hand,
\[
\kappa_{av} = \sum_{G\in O^-(v)} (x-1)^{\dim G+1},
\]
where $O^-(v)$ is taken in $P$. Since $v$ is a simple vertex, we obtain
\[
\kappa_{av} = \sum_{i=0}^{\dim P} \binom{\dim P}{i} (x-1)^{i+1} = (x-1)\cdot \sum_{i=0}^{\dim P} \binom{\dim P}{i} (x-1)^{i} = (x-1)((x-1)+1)^{\dim P} = x^{\dim P+1}-x^{\dim P}.
\]
This completes the proof.
\end{proof}

\begin{remark}
    Note that the kernel $\kappa$ is invariant with respect to replacing function $\ell$ with $-\ell$, but the characteristic kernel is not in general. Thus in part~$(2)$ of Proposition~\ref{prop:example characteristic kernel} we only have a statement about lower pyramid in contrast to part~$(2)$ of Proposition~\ref{prop:construct simple examples}.
\end{remark}
\begin{corollary}
   Let \( P \) be a polytope constructed from simple polytopes satisfying Assumption~\ref{ass:equivalent} via products and iterated lower pyramids. Then \( \CH(P) \) is simple, and \( \kappa_P \) agrees with the characteristic kernel. In particular, the \( h \)-vector of \( \CH(P) \) is \( \gamma \)-positive.
\end{corollary}
\begin{proof}
    An iterated lower pyramid over product of simplices satisfies assumptions of part $(2)$ of Proposition~\ref{prop:example characteristic kernel}. Moreover, if two polytopes $P,Q$ satisfy assumptions of part $(2)$ of Proposition~\ref{prop:example characteristic kernel}, then their product also satisfies assumptions of part $(2)$ of Proposition~\ref{prop:example characteristic kernel}. Hence the first  statement follows by induction on the number of operations.

    The statement on $\gamma$-positivity of $h$-vector of $\CH(P)$ follows from Theorem~\ref{thm:Chow poly is h of Chow} and Theorem~\ref{thm:gammapos}, since Cohen-Macauleyness of posets is preserved under operations of taking products and adding unique minimal element.
\end{proof}



\begin{example}\label{ex:nonelementary}
A family of polytopes satisfying Assumption~\ref{ass:equivalent} that do not arise as products of simplices (apart from the case of the cube) or the cone construction are trapezohedra. Below we we present a picture of tetragonal trapezohedron.

The linear function $\ell$ is the height function which is minimal on the bottom apex and maximal on the top apex. 
Denoting by $\hat 0$ the source and $\hat 1$ the sink, we have $\kappa_{\hat 0\hat 1}=(x-1)^3$. We note this is different from the characteristic kernel of this poset which is $x^3-4x^2$ on the interval between $\hat 0$~and~$\hat 1$.
\end{example}

\begin{figure}[h]
\centering
\begin{subfigure}{.5\textwidth}
  \centering
  \begin{tikzpicture}[
  scale=.8,
  line join=round, line cap=round,
  x={(0.95cm,-0.10cm)},
  y={(0.35cm, 0.32cm)},
  z={(0.00cm, 0.70cm)}
]

\def\R{2.6}
\def\h{0.85}
\pgfmathsetmacro{\a}{(3+2*sqrt(2))*\h}

\coordinate (T) at (0,0,\a);
\coordinate (B) at (0,0,-\a);

\foreach \k in {0,...,7}{
  \pgfmathsetmacro{\ang}{45*\k}
  \pgfmathtruncatemacro{\parity}{mod(\k,2)}
  \ifnum\parity=0
    \coordinate (v\k) at ({\R*cos(\ang)},{\R*sin(\ang)},{+\h});
  \else
    \coordinate (v\k) at ({\R*cos(\ang)},{\R*sin(\ang)},{-\h});
  \fi
}

\tikzset{
  edgeFront/.style={draw=black, line width=0.95pt, opacity=1},
  edgeBack/.style ={draw=black!55, line width=0.28pt, opacity=0.22},
  faintface/.style={opacity=0.16, draw=none},  
  boldface/.style ={opacity=1,    draw=none},  
}


\path[faintface, fill=gray]   (T)--(v0)--(v1)--(v2)--cycle; 
\path[faintface, fill=gray]   (T)--(v2)--(v3)--(v4)--cycle; 
\path[faintface, fill=gray]   (B)--(v1)--(v2)--(v3)--cycle; 
\path[faintface, fill=gray]   (B)--(v3)--(v4)--(v5)--cycle; 

\path[boldface, fill=blue!55]   (T)--(v4)--(v5)--(v6)--cycle; 
\path[boldface, fill=teal!55]   (T)--(v6)--(v7)--(v0)--cycle; 
\path[boldface, fill=orange!70] (B)--(v5)--(v6)--(v7)--cycle; 
\path[boldface, fill=green!60]  (B)--(v7)--(v0)--(v1)--cycle; 


\foreach \i in {0,...,7}{
  \pgfmathtruncatemacro{\j}{mod(\i+1,8)}
  \draw[edgeBack] (v\i)--(v\j);
}
\foreach \i in {0,2,4,6}{ \draw[edgeBack] (T)--(v\i); }
\foreach \i in {1,3,5,7}{ \draw[edgeBack] (B)--(v\i); }

\draw[edgeFront] (v4)--(v5);
\draw[edgeFront] (v5)--(v6);
\draw[edgeFront] (v6)--(v7);
\draw[edgeFront] (v7)--(v0);
\draw[edgeFront] (v0)--(v1);

\draw[edgeFront] (T)--(v0);
\draw[edgeFront] (T)--(v4);
\draw[edgeFront] (T)--(v6);

\draw[edgeFront] (B)--(v1);
\draw[edgeFront] (B)--(v5);
\draw[edgeFront] (B)--(v7);
\end{tikzpicture}
  \caption{A Trapezohedron $T_4$}
  \label{fig:sub1}
\end{subfigure}%
\begin{subfigure}{.5\textwidth}
  \centering
  \begin{tikzpicture}[
    scale=0.8,
    every node/.style={circle,fill=black,inner sep=1.6pt},
    edge/.style={line width=0.9pt}
]


\coordinate (B) at (0,0);

\coordinate (v1) at (-3,1.8);
\coordinate (v3) at (-1,1.8);
\coordinate (v5) at ( 1,1.8);
\coordinate (v7) at ( 3,1.8);

\coordinate (v0) at (-3,3.6);
\coordinate (v2) at (-1,3.6);
\coordinate (v4) at ( 1,3.6);
\coordinate (v6) at ( 3,3.6);

\coordinate (T) at (0,5.4);


\draw[edge] (v0)--(v1);
\draw[edge] (v1)--(v2);
\draw[edge] (v2)--(v3);
\draw[edge] (v3)--(v4);
\draw[edge] (v4)--(v5);
\draw[edge] (v5)--(v6);
\draw[edge] (v6)--(v7);
\draw[edge] (v7)--(v0);

\foreach \v in {v0,v2,v4,v6}
    \draw[edge] (T)--(\v);

\foreach \v in {v1,v3,v5,v7}
    \draw[edge] (B)--(\v);

\foreach \v in {B,v0,v1,v2,v3,v4,v5,v6,v7,T}
    \node at (\v) {};
\end{tikzpicture}
  \caption{Hasse diagram of the poset $\mathcal{O}$ of $T_4$}
  \label{fig:sub2}
\end{subfigure}
\caption{A figure with two subfigures}
\label{fig:test}
\end{figure}
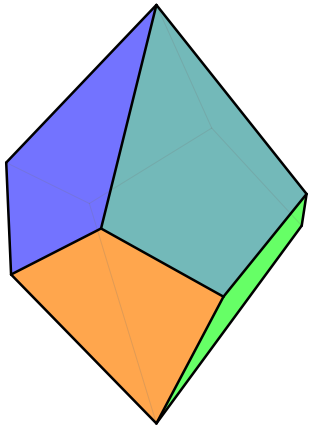
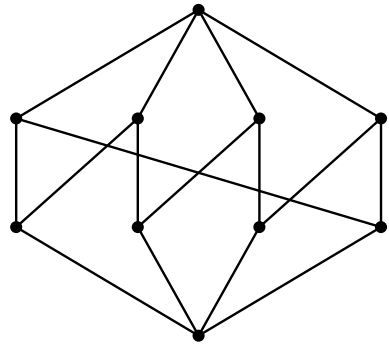

\begin{remark}
    Let us define the following classes of $n$-dimensional polytopes:
    \begin{align*}
        \mathrm{Strat}_n &= \{P \,|\, P \text{ satisfies Assumption~\ref{ass:equivalent}}\},\\
         \mathrm{SCH}_n &= \{P \,|\, P\in\mathrm{Strat}_n \text{ and } \CH(P) \text{ is simple}\},\\
         \mathrm{CharKer}_n &= \{P \,|\,   P\in\mathrm{SCH}_n \text{ and } \kappa_P \text{ is the characteristic kernel of the poset } \mathcal{O}_P\}.
    \end{align*}
By definition, we get
\[
\mathrm{CharKer}_n \subset \mathrm{SCH}_n \subset \mathrm{Strat}_n, 
\]
Example~\ref{ex:nonelementary} shows that the first inclusion is strict and Example~\ref{ex:doublepyramid} shows that the second inclusion is strict.
Moreover, Propositions~\ref{prop:constructexamples},~\ref{prop:construct simple examples}~and~\ref{prop:example characteristic kernel} provide operations on polytopes preserving classes $\mathrm{Strat}_n,\mathrm{SCH}_n$ and $\mathrm{CharKer}_n$ respectively. It would be interesting to see more operations preserving these classes as well as examples of polytopes separating them.
\end{remark}


\bibliographystyle{alpha}
\bibliography{gsv}

\end{document}